\documentclass[12pt,reqno]{amsart}
\usepackage{xypic}
\usepackage[group-separator={,}]{siunitx}

\usepackage{amssymb}
\usepackage{enumitem}
\setenumerate[1]{label=(\roman*)}
\setenumerate[0]{label=(\alph*)}
\usepackage[margin = 1 in]{geometry}
\usepackage[T1]{fontenc}
\usepackage{mathtools, multirow, tabularx, booktabs}
\usepackage{microtype}
\usepackage{hyperref}

\newcount\hour \newcount\minute \newcount\minus
\hour=\time \divide \hour by 60
\minute=\time
\minus=\hour\multiply\minus by -60
\advance \minute by \minus
\def\now{%
\number\hour:%
  \ifnum \minute<10 0\fi%
  \number\minute%
}

\newtheorem{theorem}{Theorem}[section]
\newtheorem{lemma}[theorem]{Lemma}

\newtheorem{corollary}[theorem]{Corollary}

\theoremstyle{definition}
\newtheorem{definition}[theorem]{Definition}
\newtheorem{cnj}[theorem]{Conjecture}
\newtheorem{example}[theorem]{Example}
\theoremstyle{remark}
\newtheorem{remark}[theorem]{Remark}

\newcommand{\Gammabar}{\bar{\Gamma}}
\newcommand{\mat}[1]{\begin{bmatrix} #1 \end{bmatrix}}
\newcommand{\hp}{\mathfrak{h}}
\newcommand{\fp}{\mathfrak{p}}
\newcommand{\fP}{\mathfrak{P}}

\DeclarePairedDelimiter{\pair}{\langle}{\rangle}
\DeclarePairedDelimiterX\set[1]{\lbrace}{\rbrace}{\def\given{\;\delimsize\vert\;}#1}
\DeclarePairedDelimiter{\gen}{\langle}{\rangle}
\DeclarePairedDelimiter{\size}{\lvert}{\rvert}

\newcommand{\cRR}{\mathcal{R}}
\newcommand{\Z}{\mathbb{Z}}
\newcommand{\I}{\mathbb{I}}
\newcommand{\Q}{\mathbb{Q}}
\newcommand{\R}{\mathbb{R}}
\newcommand{\C}{\mathbb{C}}

\newcommand{\F}{\mathbb{F}}
\newcommand{\PP}{\mathbb{P}}
\newcommand{\cC}{\mathcal{C}}
\newcommand{\DD}{\mathcal{D}}
\newcommand{\cB}{\mathcal{B}}
\newcommand{\inv}{^{-1}}

\def\O{{\mathcal O}}

\DeclareMathOperator{\Ker}{Ker}
\newcommand{\cusp}{\mathrm{cusp}}
\DeclareMathOperator{\GL}{GL}
\DeclareMathOperator{\SL}{SL}
\DeclareMathOperator{\diag}{diag}

\DeclareMathOperator{\image}{Im}
\DeclareMathOperator{\Gal}{Gal}
\DeclareMathOperator{\tr}{Tr}
\DeclareMathOperator{\St}{St}
\DeclareMathOperator{\ord}{ord}
\DeclareMathOperator{\id}{id}

\begin{document}
\title{Steinberg homology, modular forms, and real quadratic fields}
\author{Avner Ash}
\address{Boston College, Chestnut Hill, MA  02467}
\email{Avner.Ash@bc.edu}
\author{Dan Yasaki}
\address{UNCG, Greensboro, NC  27412}
\email{d\_yasaki@uncg.edu}
\date{\today~\now}

\begin{abstract}
We compare the homology of a congruence subgroup $\Gamma$ of $\GL_2(\Z)$ with coefficients in the Steinberg modules over $\Q$ and over $E$, where $E$ is a real quadratic field. If $R$ is any commutative base ring,  the last connecting homomorphism $\psi_{\Gamma,E}$ in the long exact sequence of homology stemming from this comparison has image in $H_0(\Gamma, \St(\Q^2;R))$ generated by classes $z_\beta$ indexed by 
$\beta\in E\setminus\Q$.  We investigate this image.

When $R=\C$, $H_0(\Gamma, \St(\Q^2;\C))$ is isomorphic to a space of classical modular forms of weight 2, and the image lies inside the cuspidal part.  In this case, $z_\beta$ is closely related to periods of modular forms over the geodesic in the upper half plane from $\beta$ to its conjugate  $\beta'$.  
Assuming GRH we prove that the image of $\psi_{\Gamma,E}$ equals the entire cuspidal part.

When $R=\Z$, we have an integral version of the situation.  
We define the cuspidal part of the Steinberg homology,
 $H_0^\cusp(\Gamma, \St(\Q^2;\Z))$.
 Assuming GRH we prove that for any congruence subgroup, $\psi_{\Gamma,E}$ always has finite index in 
 $H_0^\cusp(\Gamma, \St(\Q^2;\Z))$,
 and if $\Gamma=\Gamma_1(N)^\pm$ or $\Gamma_1(N)$,
then the image is all of
 $H_0^\cusp(\Gamma, \St(\Q^2;\Z))$.    
If $\Gamma=\Gamma_0(N)^\pm$ or $\Gamma_0(N)$,
we prove (still assuming GRH) 
an upper bound for the size of $H_0^\cusp(\Gamma, \St(\Q^2;\Z))/\image(\psi_{\Gamma,E})$.
We conjecture that the results in this paragraph are true unconditionally.

We also report on  extensive computations of the image of  $\psi_{\Gamma,E}$  that we made for 
$\Gamma=\Gamma_0(N)^\pm$ and
$\Gamma=\Gamma_0(N)$.  
Based on these computations, we believe that the image of $\psi_{\Gamma,E}$ is not all of $H_0^\cusp(\Gamma, \St(\Q^2;\Z))$ for these groups, for general $N$.  
 \end{abstract}
\maketitle

\section{Introduction}\label{intro}

In this paper, we explore a homological version of the classical concept of a toral period for a modular cusp form. 
Let $E$ be a real quadratic field and $\Gamma$ a 
congruence subgroup of $\GL_2(\Z)$, which acts on $E$ via fractional linear transformations.   Given $\beta\in E\setminus \Q$, the stabilizer $\Gamma_\beta$ modulo $\pm I$ is a cyclic group.
Let $\gamma_\beta\in\Gamma_\beta$ be a generator (modulo $\pm I$).   If $f(z)$ is a holomorphic 
modular form of weight 2 for $\Gamma$, the ``$\beta$-toral period'' of $f(z)$ is the integral
\[
\int_\tau^{\gamma_\beta \tau}  f(z) \,dz,
\]
where $\tau$ is any point in the upper half plane.  It is independent of the choice of $\tau$. 

We define a homological version of these periods over any commutative ground ring $R$, in terms of the homology of $\Gamma$ with coefficients in the Steinberg module.  It is this version that is the main object of the computations and theorems of this paper.  When $R=\C$ we prove that under the Generalized Riemann Hypothesis (GRH) the $\beta$-toral cycles generate the relevant homology group -- see  Section~\ref{complex}.  This result does not seem to be known unconditionally, nor have we seen it conjectured in the literature.
(We use the term ``GRH'' to mean the generalized Riemann hypothesis for those number fields
needed in Lenstra \cite[Theorem 3.1]{L}.) 

For general $R$,  we prove a number of theorems, detailed below, about the group generated by these cycles in homology.  We also performed extensive computations for particular groups when $R=\Z$, which are reported upon in the last portion of this paper.

Given any field $K$, let $\St(K^2;R)$ denote the Steinberg module for the vector space $K^2$ with coefficients in a ring $R$.  (In Section~\ref{mod} we review the definition and basic facts about Steinberg modules.)
  When $K=\Q$, the Steinberg module is isomorphic to the module of modular symbols 
  $[v,w]$, where $v,w$ are points in the projective line over $\Q$.  

Following~\cite{A}, we have a long exact sequence of $\GL_2(\Q)$-modules
  \begin{equation}\label{seq}
0\to \St(\Q^2;R)\to \St(E^2;R) \to C\to 0.
\end{equation}
In \cite{A} it is proven that as an $R\GL_2(\Q)$-module, $C$ is isomorphic to a direct sum of free $R$-modules of rank $1$ indexed by the elements of
$\beta\in E \setminus \Q$, where $\GL_2(\Q)$ acts via fractional linear
transformations on $E \setminus \Q$.

This short exact sequence gives rise to a long exact sequence of the homology groups
of $\Gamma$ 
with coefficients in $\St(E^2;R)$, $\St(\Q^2;R)$ and $C$. 
The main object of this paper is the connecting homomorphism $H_1(\Gamma,C)\to H_0(\Gamma, \St(\Q^2;R))$.  It is more convenient to work with its negative: 
\begin{definition}\label{psidef}
For any subgroup $\Gamma\subset\GL_2(\Q)$, let $\psi_{\Gamma,E}=-\partial$, where 
$\partial\colon H_1(\Gamma, C)\to H_0(\Gamma, \St(\Q^2;R))$ is the connecting homomorphism described above.  When the field $E$ is understood, we will just write $\psi_{\Gamma}$.
\end{definition}
We derive an explicit formula for $\psi_{\Gamma,E}$ in terms of modular symbols (Theorem~\ref{psiform}).  We define the ``cuspidal'' part of the target, 
$H_0^\cusp(\Gamma,  \St(\Q^2;R))$ and prove that the image of $\psi_{\Gamma,E}$ always lies in the cuspidal part.
On the basis of our computational evidence for $R=\Z$, presented in section~\ref{results},  we claim it is very unlikely that the image of
$\psi_{\Gamma,E}$ is the whole of the cuspidal part in general.  The question becomes to determine the image.
 
The study of the image of $\psi_{\Gamma,E}$ naturally leads us to consider
certain subgroups $\widetilde K(\Gamma,E) \subset  K(\Gamma,E)
\subset \Gamma$.  The first is the group generated by all the $\Gamma_\beta$'s, as $\beta$ ranges over $E\setminus \Q$, and the second is generated by $\widetilde K(\Gamma,E)$ together with
all the upper and lower triangular matrices in $\Gamma$.   

The connection to modular forms arises as follows:  If $6$ is invertible in $R$ and $\Gamma$ is an arithmetic group, then
 $H_0(\Gamma, \St(\Q^2;R))$ is isomorphic
by Borel-Serre duality to $H^1(\Gamma,R)$.
If $R=\C$,  $H^1(\Gamma,\C)$ is isomorphic by the Eichler-Shimura theorem to a space of modular forms of weight two for $\Gamma$.  
 
 The rest of this introduction summarizes in further detail the contents of the paper.  Define the following congruence subgroups.
 
  \begin{definition}\label{N1}  Let $N$ be a positive integer.
    
 \begin{itemize}
 \item $\Gamma_1(N)^\pm$ is the subgroup of $\GL_2(\Z)$
consisting of matrices congruent to 
\[
\mat{
\pm1&*\\
0&*
} \pmod N.
\]  

\item $\Gamma_1(N)$ is the subgroup of $\SL_2(\Z)$
consisting of matrices congruent to 
\[
\mat{
1&*\\
0&1
} \pmod N.
\]  

\item $\Gamma_0(N)^\pm$ is the subgroup of $\GL_2(\Z)$
consisting of matrices congruent to 
\[
\mat{
*&*\\
0&*
} \pmod N.
\] 

\item $\Gamma_0(N)=\Gamma_0(N)^\pm\cap \SL_2(\Z)$. 
 \end{itemize}
\end{definition}

We performed computations to find the image of $\psi_{\Gamma,E}$ (when $R=\Z$) for 
$\Gamma$ equal to 
$\Gamma_0(N)^\pm$ and $\Gamma_0(N)$ for $N\le1000$ and $E=\Q[\sqrt\Delta]$ with 
$\Delta\le50$.
We did not compute for $\Gamma$ equal to 
$\Gamma_1(N)^\pm$ and $\Gamma_1(N)$ because of the larger index of these groups in $\GL_2(\Z)$. Moreover, our theory shows that such computations would not be interesting, because for these groups we can prove (given GRH) that 
$\psi_{\Gamma,E}$ maps onto $H_0^\cusp(\Gamma,  \St(\Q^2;R))$.

It should be noted that our computations are not definitive when they do not show that 
$\psi_{\Gamma,E}$ maps onto $H_0^\cusp(\Gamma,  \St(\Q^2;R))$, because there is the possibility that computation with additional $\beta$'s would discover more elements in the image of 
$\psi_{\Gamma,E}$.  We were unable to find an effective bound on the heights of $\beta$ that would allow us to terminate the computations with perfect confidence.
However, the computations usually stabilized quite rapidly as more $\beta$'s were processed.   These computations helped us formulate our main theorems.
       
Section~\ref{mod} gives basic facts about the Steinberg module and about modular symbols.  Section~\ref{conn} derives a formula for $\psi_{\Gamma,E}$ in terms of modular symbols.   Section~\ref{stab}  determines the stabilizers $\Gamma_\beta$,
and Section~\ref{un} gives detailed information about its elements.

 Definition~\ref{defcusp} defines the 
``cuspidal submodule'' 
$H_0^{\cusp}(\Gamma, \St(\Q^2;R))\subset 
H_0(\Gamma, \St(\Q^2;R))$.  This definition is consistent with the classical definition of cusp forms if $R=\C$.
We prove (Theorem~\ref{thmcusp}) that
the image of $\psi_{\Gamma,E}$ lies in
$H_0^{\cusp}(\Gamma, \St(\Q^2;R))$.
In Section~\ref{seccusp}, we collect a number of useful results about 
$H_0(\Gamma, \St(\Q^2;R))$, $H_0^{\cusp}(\Gamma, \St(\Q^2;R))$, and modular symbols.

Our calculations for $R=\Z$ suggest that the image of $\psi_{\Gamma,E}$ always has finite index in $H_0^{\cusp}(\Gamma, \St(\Q^2;\Z))$.  We prove Theorem~\ref{maingeneral}, which says that (assuming GRH)
if $\Gamma$ is a subgroup of 
$\GL_2(\Q)$ that contains 
some principle congruence subgroup, then $H_0^{\cusp}(\Gamma, \St(\Q^2;R))$ modulo the image of
$\psi_{\Gamma,E}$ is a finitely-generated torsion $R$-module. 

When $\Gamma$ is one of the congruence subgroups defined above we have more precise results, again assuming GRH:
\begin{itemize}
\item For $\Gamma=\Gamma_1(N)^\pm$ or $\Gamma_1(N)$, 
$\psi_{\Gamma,E}$ is surjective. 

\item  For
$\Gamma=\Gamma_0(N)^\pm$, define the subgroup  $A_E(N)$ (Definition~\ref{AE}) 
of $((\Z/N\Z)^\times/\set{\pm1})$.  Then there is a surjective map $\pi$ from 
 $((\Z/N\Z)^\times/\set{\pm1})/A_E(N)$ onto
the cokernel of $\psi_{\Gamma,E}$.

\item For
$\Gamma=\Gamma_0(N)$  define the subgroup  $A_E(N)^*$ (Definition~\ref{AE}) 
of $(\Z/N\Z)^\times$.  Then there is a surjective map $\pi^*$ from 
 $(\Z/N\Z)^\times/A_E(N)^*$ onto
the cokernel of $\psi_{\Gamma,E}$.
\end{itemize}

An essential ingredient in our work is a beautiful theorem of Lenstra's, which is the source of our need to assume GRH.  Our application of Lenstra's theorem is made in Section~\ref{lenstrasection}.
Then we prove the results in the bullets in Sections~\ref{im1}, \ref{imgeneral} and \ref{im0}.   

  The key to studying the image of $\psi_{\Gamma,E}$ and proving the bulleted assertions above is the group 
 $K(\Gamma,E)$, defined earlier in this introduction.  
We prove Lemma~\ref{Kgen} which asserts that for any $\Gamma$, 
the image of $\psi_{\Gamma,E}$ is the $R$-span of $\Psi(g):=[e,ge]_\Gamma$ as $g$ runs through $K(\Gamma,E)$ and $e=(1:0)$.
Therefore, the problem of finding the image of $\psi_{\Gamma,E}$ separates into (1) determining $K(\Gamma,E)$ and (2) studying the map $\Psi\colon \Gamma\to H_0^{\cusp}(\Gamma,\St(\Q^2;R))$.  A central result about $\Psi$ is given in Theorem~\ref{Psi*}.

As for $K(\Gamma,E)$, in Theorem~\ref{maintheorem} we show using Dirichlet's theorem on primes in an arithmetic progression and  Lenstra's theorem that (assuming GRH) for any $N$ and any $E$ that $K(\Gamma_1(N)^\pm,E)=\Gamma_1(N)^\pm$
and $K(\Gamma_1(N),E)=\Gamma_1(N)$.   
These group theoretic assertions are of independent interest, and it would be nice to obtain unconditional proofs of them.

Perhaps a more natural group to study would be 
$\widetilde K(\Gamma,E)$, also defined above.  
Our proof of Theorem~\ref{maintheorem} requires a strong use of the triangular matrices and we have been unable to prove anything significant about $\widetilde K(\Gamma,E)$. Nor do we know of any general results in the literature about either $\widetilde K(\Gamma,E)$ or $K(\Gamma,E)$ .  

In Section~\ref{complex} we discuss the case $R=\C$.  We indicate how the theory of toral periods of cuspforms possibly might be employed to prove some of our theorems unconditionally for $R=\C$, but we point out that in the current state of the field, not enough is known in detail about the formulas for the periods nor about non-vanishing of $L$-functions for this approach to bear fruit at present.
Assuming GRH, we do show in Theorem~\ref{toralgen}  that the $E$-toral cycles generate the homology of the compact modular curve.  We do not know of any proof of this fact in the literature.

We describe and justify our computational methods in Section~\ref{what}.  In particular, we give an isomorphism between the cuspidal Voronoi homology of the upper half plane modulo $\Gamma$ and $H_0^{\cusp}(\Gamma, \St(\Q^2;\Z))$.  This is probably known to the experts, but we could not find it stated in the literature.

We summarize our computational results in Section~\ref{results}.  We computed for levels $N\le1000$ and $E=\Q[\sqrt \Delta]$ for $\Delta\le 50$.  Based on these results, we
conjecture in Section~\ref{conj} that the bulleted assertions above are true unconditionally.   We give details about the computation, including how the $\beta$'s are selected and how long we spend computing each image.

We find in our computations that the torsion in $H_0^{\cusp}$ is always either trivial, or isomorphic to  $\Z/3\Z$ or 
$(\Z/3\Z)^2$ for $\Gamma_0^\pm(N)$.  For $\Gamma_0^\pm(N)$, the torsion in $H_0^{\cusp}$ is always either trivial, or isomorphic to  $\Z/3\Z$ or 
$(\Z/3\Z)^3$.  Perhaps this could be proven by careful study of the long exact sequence \eqref{eq:free-6} in the proof of Theorem~\ref{free-6}.

Thanks to B.~Gross, A.~Popa, D.~Rohrlich, G.~Stevens and A.~Venkatesh for helpful suggestions concerning the material in Section~\ref{complex}.  Special thanks to K. Conrad for telling us about Lenstra's paper.  Thanks also to  R. Gross, P.~Gunnells and D. Kelmer for helpful comments.

\section{ Preliminaries on the Steinberg module and Steinberg homology}\label{mod}

For more information on the Steinberg module, see the introduction to~\cite{APS} and its references.

Let $K$ be a field, $R$ a ring, and $n\ge2$ an integer.  Let $K^n$ be the vector space of column vectors.  By definition, the Steinberg module 
$\St(K^n;R)$ is the  reduced homology of the Tits building:
\[
\St(K^n;R)=\widetilde H_{n-2}(T(K^n),R).
\]
The Tits building $T(K^n)$ is the simplicial complex with one vertex for each subvector space $V\subset K^n$ with $0\ne V\ne K^n$, where the vertices $V_1,\dots,V_k$ span a simplex if and only if they can be arranged into a flag. 
The Steinberg module is a left-module for the group ring $R\GL_n(K)$.  
When $R=\Z$ we sometimes write $\St(K^n)$ instead of $\St(K^n;\Z)$. The Steinberg module 
$\St(K^n)$ is a free $\Z$-module and $\St(K^n;R)=\St(K^n)\otimes_\Z R$.

\begin{definition}
Let $v,w\in \PP^1(K)$.
The modular symbol $[v,w]$ denotes the element in $\St(K^2;R)$ which is the
  fundamental class of the $0$-sphere which has vertices $v$ and $w$, oriented so that its boundary is $w-v$.
\end{definition}

The action of an element $g\in \GL_2(K)$ on the 
symbol $[v,w]$ for $v,w\in K^2$ is given by
\[
g[v,w]=[gv,gw],
\]
where $g$ acts on the projective line by linear fractional transformations.

We recall some standard facts about modular symbols and the Steinberg module.
The first two parts of the following theorem follow from \cite[Theorem~5]{AGM2}.
Part \ref{it:unipotent} follows  easily from \ref{it:gens} and \ref{it:relations}(b).
Compare also Cremona \cite[Proposition~2.14]{cremona} for $K = \Q$ and \cite[Proposition~5]{gunnells} for $K$ a number field.
\begin{theorem}\label{thm:mod-props} 
Let $K$ be any field.
\begin{enumerate}

\item \label{it:gens} As abelian group, $\St(K^2)$ is generated by $[v,w]$ as $v,w$ range over
  all elements of $\PP^1(K)$.

\item  \label{it:relations} The following relations hold:
\begin{enumerate}
\item $[v,w]=-[w,v]$ and in particular $[v,v]=0$ for all $v,w\in \PP^1(K)$;
\item $[v,w]=[v,x]+[x,w]$ for all $v,w,x\in \PP^1(K)$;
\end{enumerate}

\item\label{it:unipotent} Fix any element $y\in\PP^1(K)$.  Then $\St(K^2)$ has as a 
free $\Z$-basis the symbols $[y,v]$ where $v$ runs over all 
 $v\ne y \in \PP^1(K)$.  
\end{enumerate} 
\end{theorem}

There are generalizations of all these properties for $\St(K^n)$.
 
We need the following theorem.  It follows immediately from  a theorem of Bykovskii \cite[Theorem~1]{By}, who proved a similar result for $\St(\Q^n)$for all $n\ge2$.  Also, see \cite{CFP} for a new treatment of this theorem and related results for other fields than $\Q$.

\begin{theorem}\label{by}  Let $e_1,e_2$ be the standard basis of $\Z^2$.
 $\St(\Q^2)$ is isomorphic to the quotient of the free abelian group generated by symbols $\gen{a,b}$ for all $\Z$-bases $\set{a,b}$ of $\Z^2$ modulo the following relations:
  \begin{enumerate}
  \item \label{it:by-a}   $\gen{a,b}=-\gen{b,a}$ and $ \gen{-a,b}=\gen{a,b}$ for all $\Z$-bases $\set{a,b}$ of $\Z^2$;
  \item \label{it:by-b} $\gen{a,b}+\gen{-b,a+b}+\gen{a+b,-a}=0$ for all $\Z$-bases 
  $\set{a,b}$ of $\Z^2$.
  \end{enumerate}
  The isomorphism is given by $\gen{a,b} \mapsto [a',b']$ where $x'$ denotes the line through $0$ and $x$.
 \end{theorem}
 
 \begin{remark}
 Bykovskii \cite{By} states this theorem in a slightly different form.  However, using the relations in part~\ref{it:by-a}, it is easy to see that what we wrote is equivalent to his theorem when $n=2$.  Our formulation is better suited to the way we actually carry out computations with modular symbols, as detailed in Section~\ref{what}.
 \end{remark}

\begin{definition}
If $\Gamma$ is any subgroup of $\GL_n(K)$, we define the Steinberg homology of 
$\Gamma$ over $R$ to be
 $H_*(\Gamma, \St(K^n; R))$.
\end{definition}
 
\begin{definition}
If $\Gamma$ is any subgroup of $\GL_2(K)$, we set $[v,w]_\Gamma$ to be the image of $[v,w]$ in the coinvariants $H_0(\Gamma, \St(K^2; R))$.
\end{definition}
In the notation $[v,w]$ and $[v,w]_\Gamma$ we suppress mention of $R$.  The base ring $R$ will always be clear from the context.

  The following corollary follows immediately from Theorem~\ref{by}.  We need it in Section~\ref{what}.
\begin{corollary}\label{triv}
  The Steinberg homology $H_0(\Gamma, \St(K^2;R))$ is isomorphic to the $R$-module
  in Theorem~\ref{by} modulo the further relations
\[[v,w]_\Gamma=[\gamma v,\gamma w]_\Gamma \quad \text{for any $\gamma\in\Gamma$.}\]
\end{corollary}

From \cite[Proposition~VIII.8.2]{B}, it follows that $\St(\Q^n)\otimes \det$ is the dualizing module for arithmetic subgroups of $\GL_n(\Z)$ if $n$ is even, and if $n$ is odd, then  $\St(\Q^n)$ is the dualizing module.  Therefore, from the exact sequence \cite[(3.6) p. 280]{B} we obtain Brown's generalization of the Borel-Serre duality theorem:
\begin{theorem}\label{BS}
Let $\Gamma$ be a subgroup of finite index in $\GL_2(\Z)$, and let $R$  be a ring on which $6$ acts invertibly.  Then there is a natural isomorphism for $i=0,1$:
\[
\lambda \colon H_i(\Gamma, \St(\Q^2;R)) \to H^{1-i}(\Gamma, R).
\]
\end{theorem}

In particular, if $R=\C$ then we can compose $\lambda$ with the Eichler-Shimura isomorphism to obtain an isomorphism between $H_0(\Gamma, \St(\Q^2;\C))$ and a space of modular forms of weight $2$ for $\Gamma$.  We do this in Section~\ref{complex}.

\section{The connecting map \texorpdfstring{$\psi$}{psi}} \label{conn}

In this section, $G$ can be any subgroup of $\GL_2(\Q)$, and $R$ any ring.  
Until further notice, we fix the real quadratic field $E$ and write $\psi_{G}$ instead of $\psi_{G,E}$.

Let $e=(1:0)\in\PP^1(\Q)$.  We denote the vector in $\Q^2$ 
with components $(1,0)$ by $e_1$.  

From item~\ref{it:unipotent} of Theorem~\ref{thm:mod-props}, for any extension $K/\Q$, the Steinberg module 
$\St(K^2;R)$ has the $R$-basis $[e,\alpha]$ where $\alpha$ runs over  $\PP^1(K) \setminus \set{e}$.

Refer to the exact sequence~\eqref{seq} in Section~\ref{intro} for the definition of the 
$\GL_2(\Q)$-module $C$ and to Definition~\ref{psidef} for the map $\psi_G$. 
From  \cite[Theorem 5.1]{A}, we know that $C$ is a free $R$-module with $R$-basis 
$[e,b]'$, where $b\in \PP^1(E) \setminus \PP^1(\Q)$  
and $[x,y]'$  means the image of $[x,y]$ in $C$.

Let $R_b$ denote the $R$-span of $[e,b]'$.  Then 
\[
C\simeq\bigoplus_{b\in \PP^1(E) \setminus \PP^1(\Q)} R_b.
\]
This is an isomorphism of $\GL_2(\Q)$-modules where on the right hand side,   
$g\in \GL_2(\Q)$ takes the element $n$ in $R_b$ to the element $n$ in $R_{gb}$. 
(As usual, $gb$ denotes the action of $g$ on $\PP^1(E) \setminus \PP^1(\Q)$ via linear fractional transformations.)

Therefore we may view $\psi_G$ as a map
\[
\psi_G\colon H_1(G, \bigoplus_{b\in \PP^1(E) \setminus \PP^1(\Q)} R_b) \to H_0(G, \St(\Q^2;R)).
\]
By Shapiro's lemma,
\[
H_1(G, \bigoplus_{b\in \PP^1(E)\setminus \PP^1(\Q)} R_b)\simeq \bigoplus_{b\in\cB}
H_1(G_b,R),
\]
where $G_b$ is the stabilizer in $G$ of $b$,
and $\cB$ is a set of representatives of the $G$-orbits of 
$\PP^1(E)\setminus \PP^1(\Q)$.
(Compare \cite[Section~5]{A}.)

We use the bar resolution to compute the homology of a group $\Delta$ with coefficients in an $R\Delta$-module $M$.  Use the notation on page 19 of \cite{B}.
Then a $1$-chain for the homology of $\Delta$ with coefficients in $M$ is of the form
 $z=\sum [\delta]\otimes_R m_\delta\in 
 (R\Delta)^2\otimes_{R\Delta}M\simeq R\Delta\otimes_R M$.  
 The boundary map is $\partial ([\delta]\otimes_R m) = \delta m - m\in
R\Delta\otimes_{R\Delta}M\simeq M$.  Let $\cRR$ be the $R$-submodule of $M$ spanned by $\delta m - m$ as $m$ ranges over $M$.  Then
$H_0(\Delta,M) = M/\cRR=M_\Delta$, and
  $H_1(\Delta,M)$ equals $\set{z \given \sum (\delta m_\delta - m_\delta) = 0 }$ modulo $1$-boundaries.
  
Now set $M=R$, the trivial module.
Fix $b\in\cB$, and  suppose  $z=\sum [\gamma]\otimes m_\gamma$ is a $1$-cycle for 
  $G_b$ with coefficients in $R_b$.  
  Here, $\gamma$ runs through $G_b$, $m_\gamma=r_\gamma[e,x_\gamma]'\in R_b$ with 
 $ r_\gamma\in R$,
  $x_\gamma\in \PP^1(E) \setminus \PP^1(\Q)$, 
  $\sum r_\gamma(\gamma m_\gamma - m_\gamma) = 0$ and $r_\gamma=0$ for all but finitely many $\gamma$.

Note that
\[\gamma[e,x_\gamma]=[\gamma e,\gamma x_\gamma]
  =[e,\gamma x_\gamma]+[\gamma e,e]\] 
  so that
  \[\gamma m_\gamma=r_\gamma\gamma[e,x_\gamma]'=r_\gamma[e,\gamma x_\gamma]'.\]

Since $G_b$ is abelian (see Section~\ref{stab}), there is an isomorphism
  \[
  G_{b}\otimes_\Z R \to H_1(G_b, R_b)
  \]
  that sends $\gamma\otimes 1$ to $[\gamma]\otimes_R [e,b]'$, for any 
  $\gamma\in G_b$.  
  (This makes sense: $\gamma\otimes_R [e,b]'$ is a cycle because $\gamma b=b$.)
  
  \begin{theorem}\label{psiform}
    Let $b\in \PP^1(E) \setminus \PP^1(\Q)$, and let $G$ be a subgroup of $\GL_2(\Q)$. Let $G_b$ be the stabilizer of $b$ in 
    $G$.
Then for any $\gamma\in G_b$,
\[ \psi_G([\gamma]\otimes_R [e,b]')= [e,\gamma e]_G.\]
  \end{theorem}
  
    \begin{proof}  Compute the boundary map coming from the short exact sequence~\eqref{seq} as follows:
Lift $[\gamma]\otimes_R [e,b]'$ to a 1-chain for $G$ with coefficients in $ \St(E^2;R))$.
 In fact we can lift it to 
$[\gamma]\otimes_R [e,b]$.  Now take the boundary in the chain complex, obtaining
\[
(\gamma-1)[e,b] = [\gamma e,b] - [e,b]  = [e,b]+[\gamma e,e]-[e,b]
=[\gamma e,e] 
\]
and view the result modulo $1$-boundaries to obtain $[\gamma e,e]_G\in
H_0(G, \St(\Q^2;R))$.  

Since $\psi_G$ is the negative of the boundary map, this proves the theorem.
  \end{proof}
  
  From Section~\ref{stab}, we see that if $G=\Gamma$ is a congruence subgroup of 
  $\GL_2(\Z)$ then $\Gamma_b/\set{\pm I}$ is cyclic.  It is clear from the formula in the theorem that as a function of $\gamma$, 
  $ \psi_\Gamma([\gamma]\otimes_R [e,b]')$ factors through $\Gamma_b/\set{\pm I}$. Let $\gamma_b$ be an element of $\Gamma_b$ which generates $\Gamma_b/\set{\pm I}$.  Then the image of $\psi_\Gamma$ restricted to $H_1(\Gamma_b, R_b)$ is generated over $R$ by 
  $[e,\gamma_b e]_\Gamma$.  
  (Note that $[e,\gamma_b^k e]_\Gamma=[e,\gamma_b e]_\Gamma+[\gamma_be,\gamma_b^2 e]_\Gamma+\cdots+[\gamma_b^{k-1} e, \gamma_b^ke]_\Gamma]=
  k[e,\gamma_b e]_\Gamma$
  and $[e,\gamma_b\inv e]_\Gamma=[\gamma_b e,e]_\Gamma
  =-[e,\gamma_b e]_\Gamma$.)
  Therefore:
  
   \begin{corollary} \label{cor:im_psi}
   Let $\Gamma$ be a congruence subgroup of $\GL_2(\Z)$.
 Then the image of $\psi_\Gamma$ is the $R$-span of the symbols 
 $[e,\gamma_b e]_\Gamma$, where  $b$ 
 runs over $\PP^1(E) \setminus \PP^1(\Q)$.
  \end{corollary}
  
  The image of $\psi_G$ depends only on a certain subgroup of $G$.  
The formula for $\psi_G$ given in Theorem~\ref{psiform} suggests we make the following definition.

\begin{definition}\label{Psidef}
Let $K\subset L\subset \GL_2(\Q)$ be subgroups, and let 
$\Psi_{K,L}\colon K\to H_0(L,\St(\Q^2;R))$ be given by $k\mapsto [e,ke]_L$.  
\end{definition}
  
  \begin{lemma}\label{Psilemma1}
Let $K\subset L\subset \GL_2(\Q)$ be subgroups.  Then  
$\Psi_{K,L}$ is a group homomorphism.
\end{lemma}

  \begin{proof}
    For any $x, y \in K$, 
\[\Psi_{K,L}(xy)= [e,xye]_L=[e,xe]_L+[xe,xye]_L=
[e,xe]_L+[e,ye]_L=\Psi_{K,L}(x)+\Psi_{K,L}(y).\]
\end{proof}
  
   \begin{lemma}\label{S}
Let $S$ be a subset of $G$, and let $\gen{S}$ be the subgroup of $G$ generated by $S$.  Then the $R$-spans in
 $H_0(G, \St(\Q^2;R))$ of 
$\set{[e, se]_G \given s\in S }$ and $\set{[e, te]_G \given t\in \gen{S} }$ are the same.
   \end{lemma}
   
      \begin{proof}
   The $\Z$-span of $\set{[e, te]_G \given t\in \gen{S} }$ is the same as the image of $\Psi_{\gen{S},G}$.  By Lemma~\ref{Psilemma1}, this is the same as the $\Z$-span of $\set{[e, se]_G \given s\in S }$.  Similarly, the $R$-spans are also the same.

   \end{proof}
     
    \begin{corollary}\label{widetilde}
 The image of $\psi_G$ is the span of the symbols $[e,ge]_G$, where 
 $g$ runs over the group $\widetilde K(G)$ generated by the stabilizers of $b$ in $G$, where $b$ runs over $\PP^1(E) \setminus \PP^1(\Q)$.
   \end{corollary}

\section{The stabilizers \texorpdfstring{$\Gamma_\beta$}{Gammabeta}}\label{stab}

Let $(\beta:1)\in \PP^1(E) \setminus \PP^1(\Q)$.  Define the homomorphism 
$\rho_\beta \colon E^\times\to\GL_2(\Q)$ by
\[
\rho_\beta(x)
\begin{bmatrix}
\beta\\
1
\end{bmatrix}=
x
\begin{bmatrix}
\beta\\
1
\end{bmatrix}.
\]

Conversely, suppose $g\in\GL_2(\Q)$ stabilizes  $(\beta:1)\in \PP^1(E) \setminus \PP^1(\Q)$.  Then there exists $x(g)\in E^\times$ such that 
\[
g
\begin{bmatrix}
\beta\\
1
\end{bmatrix}=
x(g)
\begin{bmatrix}
\beta\\
1
\end{bmatrix}.
\]
Clearly the map $g\mapsto x(g)$ is an injective homomorphism from the stabilizer of $(\beta:1)$ in $\GL_2(\Q)$ to $E^\times$.  Therefore, this stabilizer is abelian.  We can say more about the intersection of this stabilizer with a subgroup of 
$\GL_2(\Z)$:

\begin{theorem}
Let $\epsilon$ be the fundamental unit of $E$.
Let $(\beta:1)\in \PP^1(E) \setminus \PP^1(\Q)$, and let $\Gamma\subset\GL_2(\Z)$ be a congruence subgroup. Let $\Gamma_\beta$ be the stabilizer of $\beta$ in $\Gamma$.
  The quotient $\Gamma_\beta/\set{\pm I}$ is the infinite  cyclic group  generated by the matrix 
$\gamma_\beta$ (modulo $\pm I$) such that 
\[
\gamma_\beta
\begin{bmatrix}
\beta\\
1
\end{bmatrix}=
\epsilon_\beta
\begin{bmatrix}
\beta\\
1
\end{bmatrix},
\]
where $\epsilon_\beta$ is the smallest positive power of $\epsilon$ such that 
 $\rho_\beta(\epsilon_\beta)\in\Gamma$.

\end{theorem}

\begin{proof}
Consider 
\[
\gamma=
\begin{bmatrix}
A&B\\
C&D
\end{bmatrix}
\in\GL_2(\Z).
\]
Then $\gamma$ stabilizes $(\beta:1)$ if and only if there exists $x\in E$ such that
\[
\gamma
\begin{bmatrix}
\beta\\
1
\end{bmatrix}=
x
\begin{bmatrix}
\beta\\
1
\end{bmatrix}.
\]
Since $(\beta,1)$ is a $\Q$-basis of $E$, the minimal polynomial of $x$ is
\[x^2-\tr(\gamma)x+\det(\gamma)=x^2-(A+D)x + (\pm 1) = 0.\]
So $x$ is a unit in the ring of integers $\O_E$, and modulo $\pm 1$ it is equal to a power of $\epsilon$.  The result is now clear, except for showing that some power $\epsilon^k$ satisfies $\rho_\beta(\epsilon^k)\in\Gamma$.

Let $\rho=\rho_\beta(\epsilon)$.  We must show $\rho^k\in\Gamma$ for some $k$.
Write $\beta=\sigma/\tau$ for $\sigma,\tau\in \O$, the ring of integers in $E$.  We have 
\[
\rho
\begin{bmatrix}
\sigma\\
\tau
\end{bmatrix}=
\epsilon
\begin{bmatrix}
\sigma\\
\tau
\end{bmatrix}.
\]

Let $L$ be the $\Z$-lattice $\Z \sigma+\Z \tau $.  It  is in $\O$ and has $\Z$-rank 2 (because $\beta\not\in E$)
and therefore there exists some rational integer $M>0$ such that $\O\supset L\supset M\O$.  Then $\epsilon$ acts by multiplication on the finite set $\O/M\O$ and therefore for some $m>0$, $\epsilon^m\sigma\equiv \sigma$ and 
 $\epsilon^m\tau\equiv \tau$ modulo $M\O$.  Therefore 
 $(\epsilon^m-1)\sigma\in M\O\subset L$ and 
  $(\epsilon^m-1)\tau\in M\O\subset L$.  In other words
  \[
  (\rho^m-I)
\begin{bmatrix}
\sigma\\
\tau
\end{bmatrix}=
(\epsilon^m-1)
\begin{bmatrix}
\sigma\\
\tau
\end{bmatrix}=W
\begin{bmatrix}
\sigma\\
\tau
\end{bmatrix}
\]
for some $W\in M_2(\Z)$.  It follows that $\rho^m-I=W\in M_2(\Z)$ and therefore
$\rho^m\in M_2(\Z)$.  Because $\det\rho^m=\pm1$, we obtain that 
$\rho^m\in \GL_2(\Z)$.
 Then some further power of $\rho^m$ is in $\Gamma$ because $\Gamma$ contains $\Gamma(N)$ for some $N$ and $\GL_2(\Z)/\Gamma(N)$ is finite.
\end{proof}

\section{Steinberg homology and cuspidal Steinberg homology}\label{seccusp}

\begin{definition}  For a ring $R$, 
let $\DD(R)$ denote the group of divisors on $\PP^1(\Q)$ with coefficients in $R$.  For $v\in\PP^1(\Q)$, let $(v)$ denote the corresponding cusp.
Define the $\GL_2(\Q)$-module homomorphism 
$\partial\colon \St(\Q^2;R)\to \DD(R)$ by $\partial([u,v])=(v)-(u)$.  When $R$ is clear from the context,  we may write $\DD$ instead of $\DD(R)$.
\end{definition}

One checks easily that $\partial$ is well defined, using Theorem~\ref{thm:mod-props}.

Now fix a ring $R$.
For any $G\subset\GL_2(\Q)$, $\partial$ induces a map
\[
\partial \colon H_0(G,\St(\Q^2;R))=\St(\Q^2;R)_G\to H_0(G,\DD)=\DD_G.
\]

\begin{definition}\label{defcusp}
  The \emph{cuspial homology} $H_0^{\cusp}(G, \St(\Q^2;R))$ is the kernel of $\partial$.
\end{definition}

\begin{theorem}\label{thmcusp}
 Let $G$ be a subgroup
 of $\GL_2(\Q)$.
  \begin{enumerate}
  \item  The image of $\psi_G$ lies in $H_0^{\cusp}(G,\St(\Q^2;R))$.  
  \item The image of $\psi_G$ does not depend on the choice of a base point $e$.
If $e$ and $f$ are any two points in $\PP^1(\Q)$ and $\gamma\in G$, then
 $[e,\gamma e]_G= [f,\gamma f]_G$.
  %\item The image of $\psi_G$ is stable under the action of the Hecke operators.
  \end{enumerate}
\end{theorem}

\begin{proof} \mbox{}
  \begin{enumerate}
  \item  For any $\gamma \in G$,   $\partial [e,\gamma e]_G=\gamma e - e = 0$ in $\DD_G$.
  \item  For any $\gamma \in G$, \[[f,\gamma f]_G =
    [f,e]_G+[e,\gamma e]_G +[\gamma e,\gamma f]_G = [e,\gamma e]_G .\]
  \end{enumerate}
\end{proof}

\begin{lemma}\label{hcusp} 
  Let $G$ be a subgroup of $\GL_2(\Q)$, and let $\cC$ be a set of representatives of the $G$-orbits in $\PP^1(\Q)$ containing $e=(1:0)$.  Let $W$ be the span of the modular symbols of the form $[e,c]_G$ where $c$ runs over $\cC$, and let $V$ be the $R$-span of the modular symbols of the form $[e,ge]_G$ where $g$ runs over $G$. 
  Then
  \begin{equation}\label{eq:hcusp-a}
    H_0(G,\St(\Q^2;R))=H^{\cusp}_0(G,\St(\Q^2;R)) \oplus W
  \end{equation}
  and
  \begin{equation}\label{eq:hcusp-b}
    H^{\cusp}_0(G,\St(\Q^2;R)) = V.    
  \end{equation}
\end{lemma}

\begin{proof}
It is clear that $V$ is contained in $H^{\cusp}_0(G,\St(\Q^2;R))$.  
Therefore to prove both \eqref{eq:hcusp-a} and \eqref{eq:hcusp-b}, it suffices to show that   $H_0(G,\St(\Q^2;R))=V+W$ and $H^{\cusp}_0(G,\St(\Q^2;R)) \cap W = 0$.  

Now $H_0(G,\St(\Q^2;R))$ is generated over $R$ by symbols of the form 
$[e,b]_G$, where $b$ runs over $\PP^1(\Q)\setminus \set{e}$. 
Choose any $b$ and show that $[e,b]_G$ is in $V + W$:  
 
 Given $b$, there is an  $h$ in $G$ such that $b = hc$, for some $c$ in $\cC$. Then 
\[
[e,he]_G=[e,b]_G+[hc,he]_G=[e,b]_G+[c,e]_G.
\]
So
$[e,b]_G=[e,he]_G+[e,c]_G$ is in $V + W$.  Thus $H_0(G, \St(\Q^2;R))=V+W$.

Now suppose $w=\sum_{c\in \cC} r_c[e,c] \in H^{\cusp}_0(G,\St(\Q^2;R)) \cap W$, for some $r_c\in R$.  Since $[e,e]_G=0$, without loss of generality $r_e=0$.
Then
\[0=\partial w = \sum_{c\in \cC\setminus \set{e}} r_c((c)_G-(e)_G),\]
where $(x)_G$ denotes the cusp $x$ modulo $G$.  Since $\cC$ consists of $G$-inequivalent cusps, $r_c=0$ for all $c$ and $w=0$.
\end{proof}

Now let $\Gamma$ be a subgroup of finite index in $\GL_2(\Z)$.  Let $Y$ denote the upper half plane.

\begin{theorem}\label{free-6}
Let $R$ be a PID, and let $\Gamma$ be a subgroup of finite index in $\GL_2(\Z)$.
\begin{enumerate}
\item 
 $H_0(\Gamma, \St(\Q^2;R))$ and $H_0^{\cusp}(\Gamma, \St(\Q^2;R))$  are finitely generated $R$-modules.

\item The $R$-torsion in 
$H_0(\Gamma, \St(\Q^2;R))$ and $H_0^{\cusp}(\Gamma, \St(\Q^2;R))$ is annihilated by some power of $6$. 

\item If  $6$ is invertible in $R$, then $H_0(\Gamma, \St(\Q^2;R))$ and $H_0^{\cusp}(\Gamma, \St(\Q^2;R))$ are free $R$-modules of finite rank.
\end{enumerate}
\end{theorem}
 
 \begin{proof}
Since $Y(\Gamma)=Y/\Gamma$ has the homotopy type of a finite graph, $H^1(Y(\Gamma),R)$ is a  free $R$-module of finite rank.   
  The stabilizers of $\Gamma$ on $Y$ are finite groups of orders dividing 6.  It follows from the spectral sequence 
  \cite[(7.10) p. 174]{B} that the $R$-torsion in $H^1(\Gamma,R)$ is annihilated by $6$.
  
From the same spectral sequence, if $6$ is invertible in $R$, then  
  $H^1(Y(\Gamma),R)=H^1(\Gamma,R)$ so by Theorem~\ref{BS}, $H_0(\Gamma, \St(\Q^2;R))$ and therefore also $H_0^{\cusp}(\Gamma, \St(\Q^2;R))$ are free $R$-modules of finite rank.  This proves the third statement.
  
From~\cite[(3.6) p. 280]{B}, we have an exact sequence 
  \begin{equation} \label{eq:free-6}
  H^0(\Gamma,R)\to \widehat H^0(\Gamma,R) \to H_0(\Gamma, \St(\Q^2;R))
  \to   H^1(\Gamma,R),
  \end{equation}
  where $\widehat H$ denotes Farrell cohomology.  
  The cohomology of $\Gamma$ with trivial coefficients $R$ is a finitely generated $R$-module in each degree, and so is the Farrell cohomology.  (For the latter, use for example~\cite[Exercise (5)(b) p. 281.]{B})  This implies the first statement. 
  
  By \cite[Exercise (2) p. 280]{B}  and \cite[Lemma IX.9.2]{B}, we see that $\widehat H^0(\Gamma,R)$ is 
  annihilated by some power of $6$.   Together with the first paragraph of the proof, this implies the second statement.
   \end{proof}
 
In particular, if $R=\Z$, we have:
  \begin{corollary}\label{ztorsion}
The homology groups  $H_0(\Gamma, \St(\Q^2;\Z))$ and $H_0^{\cusp}(\Gamma, \St(\Q^2;\Z))$ are finitely generated 
 $\Z$-modules whose torsion modules are annihilated by a power of $6$.
   \end{corollary}
   
   \begin{remark}  
From Section~\ref{complex}, it follows that 
 $H_0^{\cusp}(\Gamma_0(N), \St(\Q^2;\Z))$ modulo torsion has rank $2g_0(N)$
and Corollary~\ref{dims} implies that  $H_0^{\cusp}(\Gamma_0(N)^\pm, \St(\Q^2;\Z))$ modulo torsion has rank $g_0(N)$, where $g_0(N)$ is the genus of the modular curve $X_0(N)$.  We observe these ranks in our computations, which gives a   check on their correctness.
\end{remark}
  
Recall Definition~\ref{Psidef}.

\begin{lemma}\label{Psilemma2}
Let $K\subset \GL_2(\Q)$ be a  subgroup.  Then the $R$-span of the image of 
 $\Psi_{K,K}:K\to H^{\cusp}_0(K, \St(\Q^2;R))$ is all of $H^{\cusp}_0(K, \St(\Q^2;R))$. 
\end{lemma}

\begin{proof}
 This follows immediately from \eqref{eq:hcusp-b} of Lemma~\ref{hcusp}.
\end{proof}

\begin{theorem}\label{Psi*}
Let $K\subset L\subset \GL_2(\Q)$ be subgroups, and let $\widehat K_L$ denote the normal closure of $K$ in $L$.  
Set $Q=L/\widehat K_L$ and let $R[\image(\Psi_{K,L})]$ denote the $R$-span of the image of $\Psi_{K,L}$.
Then the quotient 
\[
X=H^{\cusp}_0(L, \St(\Q^2;R))/R[\image(\Psi_{K,L})]
\]
 is isomorphic as $R$-module to a quotient of $Q^{ab}\otimes_\Z R$.
\end{theorem}

\begin{proof}  
For brevity write $\Psi_{L,L}=\Psi$ and for any subgroup $A$ of $L$ write
 $\Psi_A=\Psi |_A=\Psi_{A,L}$.
Then $\Psi\colon L\to H_0^{\cusp}(L,\St(\Q^2;R))$ is given by 
$\Psi(g)=[e,ge]_L$.
By Lemma~\ref{Psilemma1}, $\Psi$ is a group homomorphism.
 
Consider the commutative diagram:
\[
\xymatrix{
K\ar[r]\ar[d]_{\Psi_K}&\widehat K_L\ar[r]\ar[dl]_{\Psi_{\widehat K_L}}&L\ar[r]\ar@{->>}[dll]^\Psi&Q\ar@{-->>}[ddlll]^\theta
\\
H_0^{\cusp}(L,\St(\Q^2;R))\ar@{->>}[d]^\pi
\\
X
}
\]
By Lemma~\ref{Psilemma2}, the image of $\Psi$ generates its target over $R$.  Therefore the $R$-span of the image of $\pi\circ\Psi$ equals
$X$.

 The group $\widehat K_L$ is generated by elements of the form $\ell k\ell\inv$ for 
$\ell\in L$ and $k\in K$.  Because the target of $\Psi$ is abelian, 
$\Psi(\ell k\ell\inv)=\Psi(k)$.   Thus the image of ${\Psi_{\widehat K_L}}$ and the image of $\Psi_K$ are the same.  

Define the map $\theta$ as follows:  Given any $\bar x\in Q$, let $x$ be a lift of it to $L$.  Set $\theta(\bar x)=\pi(\Psi(x))$.
  Since (by the preceding paragraph) $\pi\circ\Psi_{\widehat K_L}=0$, this does not depend on the choice of $x$.  Check that $\theta$ is a homomorphism:  we may choose lifts so that $\bar x\bar y$ is lifted to $xy$.  Then
\[\theta(\bar x\bar y)=(\pi\circ\Psi)(xy)=(\pi\circ\Psi)(x)+(\pi\circ\Psi)(y)=\theta(\bar x)+\theta(\bar y).\]  

The homomorphism $\theta$ factors through $Q^{ab}$ because the target is abelian.  Because $\theta=\pi\circ\Psi$, the $R$-span of the image of $\theta$ is all of $X$.  Therefore $\Z$-module map 
$\theta^{ab} \colon Q^{ab}\to X$ extends to an $R$-module map
$\theta^{ab}\otimes_\Z R \colon Q^{ab}\otimes_\Z R\to X$ which is surjective.

\end{proof}

\section{Unital matrices}\label{un}

To determine the image of $\psi_\Gamma$, we need to understand the group generated by the stabilizers $\Gamma_\beta$ (whose members we call ``unital matrices'') and the triangular matrices in $\Gamma$.  In this section we begin a detailed investigation of the unital matrices in $\Gamma$.

Fix a real quadratic field $E=\Q(\sqrt\Delta)$ with 
$\Delta$ a squarefree positive integer and with fundamental unit 
$\epsilon$.  
Let the ring of integers $\O_E$ be generated over $\Z$ by $1$ and 
$\omega$.
 Let a superscript prime denote the Galois conjugate of an element in $E$.

\begin{definition}\label{unital}
Let $\beta\in E\setminus\Q$.
We say that $\gamma\in\GL_2(\Z)$ is \emph{$\beta$-unital} if 
$\gamma(\beta:1)=(\beta:1)$.  If $\gamma$ is 
$\beta$-unital for some $\beta\in E\setminus\Q$, we say $\gamma$ is \emph{unital}.
\end{definition}

By Section~\ref{stab}, $\gamma$ is $\beta$-unital if and only if there exists an integer $k$ such that 
\[
\gamma\begin{bmatrix}
\beta\\
1
\end{bmatrix}
=\pm\epsilon^k\begin{bmatrix}
\beta\\
1
\end{bmatrix}.
\]
In this case, $\gamma$ equals $\pm\gamma_\beta^k$.

\begin{lemma}\label{conjugate}
Let $\gamma_1,\gamma_2$ and $\gamma$ be $\beta$-unital, and let $g\in\GL_2(\Q)$.  Then 
\begin{enumerate}
\item $\gamma\inv$ and $\gamma_1\gamma_2$ are  $\beta$-unital; 
\item $h=g\gamma g\inv$ is $\beta^*$-unital, where $g(\beta:1)=(\beta^*:1)$.
\end{enumerate}
\end{lemma}

\begin{proof}
The first statement follows immediately from Section~\ref{stab}.  For the second statement, note that 
\[
h(\beta^*:1)=g\gamma g\inv g(\beta:1)=g\gamma(\beta:1)=
g(\beta:1)=(\beta^*:1)
.\]
\end{proof}
Clearly, $h$ and $g$ correspond to the same power of the fundamental unit $\pm\epsilon^k$.

\begin{theorem}\label{main1}
Let  $p$ be an odd prime number such that $(p)=\fp\fp'$ splits in $E$ and $a\in\Z$ prime to $p$ such that $a$ is congruent to some unit in $\O_E^\times$ modulo $\fp$.  Then there exists 
$\beta\in E\setminus \Q$ depending only on $p$ and $\Delta$,
and a $\beta$-unital matrix $M(a)\in\GL_2(\Z)$ of the form
\[
M(a)=\begin{bmatrix}
a+ps&*\\
pt&*
\end{bmatrix}.
\]
Moreover, for $\lambda=e+f\sqrt\Delta\in\O_E^\times$ such that 
$\lambda\equiv a \pmod \fp$, we may find $M(a)$ such that its determinant equals the norm of $\lambda$, 
$t=2f$ and $a+ps=x_0f+e$ for some integer $x_0$, which is odd if both $e$ and $f$ are half-integers.
\end{theorem}

\begin{proof}  
Let an unknown unital $\gamma$ be
\[
\gamma=\begin{bmatrix}
A&B\\
C&D
\end{bmatrix}
\]
whose unknown $\beta$ is 
\[
\beta=x+y\sqrt\Delta
\]
with $x,y\in\Q$.

Fix $\lambda\in\O_E^\times$ such that 
$\lambda\equiv a \bmod \fp$. Write
\[
\lambda=e+f\sqrt\Delta
\]
with $e,f\in\Z$, unless $\Delta\equiv 1 \bmod 4$, in which case $e,f$ may be a pair of half-integers. 

We solve
\[
\gamma\begin{bmatrix}
\beta\\
1
\end{bmatrix}
=\lambda\begin{bmatrix}
\beta\\
1
\end{bmatrix}
\]
with the added constraint that $p|C$ and $A\equiv a \bmod p$.  Note that the determinant of $\gamma$ equals the norm of $\lambda$.  After we solve this problem, we will take $M(a)=\gamma$.

This equation is equivalent to 
\[
\begin{bmatrix}
\beta&\beta'\\
1&1
\end{bmatrix}
\begin{bmatrix}
\lambda&0\\
0&\lambda'
\end{bmatrix}
\begin{bmatrix}
\beta&\beta'\\
1&1
\end{bmatrix}\inv=
\begin{bmatrix}
A&B\\
C&D
\end{bmatrix}
\] 
or
\[
\frac{1}{\beta-\beta'}
\begin{bmatrix}
\lambda\beta-\lambda'\beta'&-(\lambda-\lambda')N(\beta)\\
\lambda-\lambda'&-\lambda\beta'+\lambda'\beta
\end{bmatrix}=
\begin{bmatrix}
A&B\\
C&D
\end{bmatrix}
\]
where $N$ denotes the norm from $E$ to $\Q$.

The left-hand side is in $\GL_2(\Q)$  and has determinant $\pm1$.  We need to choose 
$\beta$ so that $A,B,C,D$ are integers, $p|C$ and $A\equiv a \bmod p$.  

Note that if $z=s+t\sqrt\Delta$ then $z'=s-t\sqrt\Delta$ and $z-z'=2t\sqrt\Delta$ and $Nz=s^2-t^2\Delta$.  Thus we have
\[
\beta-\beta'=2y\sqrt\Delta,\ \ \lambda-\lambda'=2f\sqrt\Delta,\ \ N(\beta)=x^2-y^2\Delta
\]
and
\[
\beta\lambda=(ex+fy\Delta)+(xf+ey)\sqrt\Delta,\ \ \beta\lambda-\beta'\lambda' = 
2(xf+ey)\sqrt\Delta
\]
and 
\[
\beta\lambda'=(ex-fy\Delta)+(-xf+ey)\sqrt\Delta,\ \ \beta\lambda'-\beta'\lambda=
2(-xf+ey)\sqrt\Delta.
\]

Our requirements, expressed in terms of $x,y,e,f$ become
\begin{align*}
A&=\frac{\beta\lambda-\beta'\lambda' }{\beta-\beta'}=\frac{xf+ey}{y}=\frac{xf}{y}+e
\in\Z \quad {\text{and}} \quad A \equiv a \bmod p;\\
C&=\frac{\lambda-\lambda'}{\beta-\beta'}=\frac{f}{y}\in\Z \quad {\text{and}} \quad   C \equiv 0 \bmod p;\\
B&=-N(\beta)C=-(x^2-y^2\Delta)\frac{f}{y}\in\Z; \quad \text{and}\\
D&=\frac{\beta\lambda'-\beta'\lambda }{\beta-\beta'}=\frac{-xf+ey}{y}\in\Z.
\end{align*}

We have fixed  a unit $\lambda=e+f\sqrt\Delta\in\O_E^\times$ such that 
$\lambda-a\in\fp$.   Choose $y=1/(2p)$, $x=x_0/(2p)$  where $x_0\in\Z$ satisfies  $\sqrt\Delta - x_0\in \fp$.  Such $x_0$ exists because $p$ is split in $E$.  We can always add $p$ to $x_0$, so without loss of generality, we may assume that $x_0$ has the same parity as $\Delta$. 

Now $C=2pf$ is an integer and divisible by $p$.  

Next $B=-\frac{x_0^2-\Delta}{2p}f$.  This
 is an integer because 1) $x_0$ and $\Delta$ have the same parity; 2)
$p$ divides  $x_0^2-\Delta$;
and 3) if $f$ is a half-integer,  $\Delta\equiv1\ $(mod 4), $x_0$ is odd, and so
$4$ divides  $x_0^2-\Delta$.

Now we come to $A=x_0f+e$.  This is an integer because $e$ and $f$ are either both integers or both half integers and in the latter case $\Delta$ is odd and therefore $x_0$ is odd.  We need $A$ to be congruent to $a$ modulo $p$.  We have chosen $\lambda=e+f\sqrt\Delta$ such that 
$\lambda-a\in\fp$.  
Now $\sqrt\Delta\equiv x_0\bmod \fp$ and $\lambda\equiv a\bmod\fp$.  Therefore 
$e+x_0f-a\equiv 0\bmod \fp$.  But $e+x_0f-a$ is an integer, and 
it is contained in $\fp$.  Since $\fp \cap \Z=(p)$, $A$ is congruent to $a$.

Then we see that $D=-x_0f+e$ is also an integer, because $D+A=2e\in\Z$.  

We have found a unital matrix $M(a)$ of the form 
\[
\begin{bmatrix}
a+ps&B\\
pt&D
\end{bmatrix}
\]
for some integers $s$ and $t$.  Note that $pt=C=p(2f)$ and $A=x_0f+e$.  Its determinant is the norm of $\lambda$.

To see that $\beta$ is independent of $a$, notice that $x$ and $y$ depend only on $p$ and $\Delta$.
\end{proof}

\begin{definition}\label{N2}
Given a positive integer $M$,
let $U_M^\pm$ denote the set of units $\lambda\in \O_E^\times$ which are congruent to $\pm1$ modulo $M$.  That is to say, $\lambda\in\pm1+M\O_E$.
\end{definition}

\begin{definition}\label{tri}
Say that $\gamma\in\GL_2(\Q)$ is \emph{triangular} if 
\[
\gamma=
\begin{bmatrix}
\pm1&0\\
*&\pm1
\end{bmatrix}\  \ \   { \text{or}} \ \ \ \ 
\begin{bmatrix}
\pm1&*\\
0&\pm1
\end{bmatrix}.
\]
\end{definition}

\begin{definition}\label{N3}
For any subgroup $\Gamma$ of $\GL_2(\Q)$, let $K(\Gamma,E)$ be the group generated by  
 all unital matrices in $\Gamma$ and all triangular matrices in 
 $\Gamma$.  If $E$ is understood, we just write $K(\Gamma)$.

\end{definition}

\begin{theorem}\label{main2}
Let $N$ be a positive integer, and
 let  $p$ be an odd prime number not dividing $N$ and such that $(p)=\fp\fp'$ splits in $E$ and the reduction of 
$U_{4N}^\pm$ modulo $\fp$  equals
$(\O_E/\fp)^\times\simeq(\Z/p\Z)^\times$. Let $m\in\Z$ be prime to $2p$ and such that $m\equiv\pm1$ (mod $N$).   Then  
 there exists $M_N(m,p)$ in $K(\Gamma_1(N)^\pm)$ such that
\[
M_N(m,p)=\begin{bmatrix}
m&b\\
pN&d
\end{bmatrix} 
\]
for some $b,d\in\Z$.
\end{theorem}

\begin{proof}
Choose a unit  $\lambda\in U_{4N}^\pm$ congruent to $m$ modulo $\fp$.   
The norm of $\lambda$ is congruent to $1$ modulo $4$ and therefore equals $1$.
Both $\lambda$ and $m$ are congruent to $\pm 1$ modulo $N\O_E$.  First,  assume that 
$\lambda\equiv m\bmod N\O_E$.  
Since $p$ and $N$ are relatively prime, we have
$\lambda\equiv m\bmod N\fp$.  

By theorem~\ref{main1}, there exists a unital matrix 
\[
M=
\begin{bmatrix}
m+ps&B\\
p(2f)&D
\end{bmatrix}
\in \GL_2(\Z),\]
where $\lambda = e + f \sqrt\Delta$ and $m+ps=x_0f+e\in\Z$
for some $x_0\in\Z$.  Furthermore, the determinant of $M$ equals the norm of 
$\lambda$, so that $\det(M)=1$.
 
Then
$\lambda\in U_{2N}^\pm$ implies that $(e-u)+f\sqrt\Delta = 2N(e'+f'\sqrt\Delta) $ for some $u\in \set{1, -1}$ and some $e',f'$ either both rational integers or both half rational integers. 
We see that $e-u$ and $f$ are in $\Z$ and divisible by $N$.    
Also, $\lambda\equiv m \equiv u \bmod N$.

Now $ps=x_0f+(e-m)\in\Z$ and $m\equiv u\bmod N$.  So $N$ divides both $e-m$ and $f$, and therefore $N$ also divides $ps$.  Since $p$ and $N$ are relatively prime, $N\mid s$.

Conjugate $M$ by $\diag(1,N/(2f))$ on the left.  
We obtain, by Lemma~\ref{conjugate}, a unital matrix
\[
M'=
\begin{bmatrix}
m+ps&B'\\
pN&D
\end{bmatrix}\in\GL_2(\Z).
\]

Premultiply $M'$ by the triangular matrix
\[
\begin{bmatrix}
1&-s/N\\
0&1
\end{bmatrix} \in \SL_2(\Z).
\]
 We obtain a matrix in 
$K(\Gamma_1(N^\pm))$:
\[
M^*_N(m,p)=
\begin{bmatrix}
m&b\\
pN&d
\end{bmatrix}.
\]

We are now finished in the case, $\lambda\equiv m\bmod N$, taking $M_N(m,p)=M_N^*(m,p)$.  On the other hand, if
$\lambda\equiv -m\bmod N$, take $M_N(m,p)=-sM_N^*(-m,p)s$, where
$s=\diag(1,-1)$. (Note that $s$ and $-I$ are triangular and in $\Gamma_1(N)^\pm$.)
\end{proof}

 In our initial computations, we found that for $\beta\in \O_E$, the image of  
  $\psi_\Gamma$ on $\gamma_\beta$ always vanished, when 
  $\Gamma=\Gamma_0(N)^\pm$.  In the next theorem, we prove that this must be the case.  However, when $\Gamma=\Gamma_0(N)$, there are many examples where 
$\psi_\Gamma(\gamma_\beta)\ne0$.  
For instance for $N = 11$,  if $E = \Q(\sqrt2)$ and $\beta = \sqrt2$, or if
  $E =  \Q(\sqrt5)$ and  $\beta = (1 + \sqrt5)/2$, 
  then $\psi_{\Gamma_0(N),E}(\gamma_\beta)\ne0$.

\begin{theorem}\label{thm:integralbeta}
Let $\Gamma=\Gamma_0(N)^\pm$, and let $\beta\in E$ be integral.  
Then $\psi_\Gamma(\gamma_\beta)=0$.  
\end{theorem}

\begin{proof}  Let $y$ be in $\O_E\setminus\Z$ and $g$ be an element in $\GL_2(\Z)$ such that 
\[
g\begin{bmatrix}
y\\
1
\end{bmatrix}=\eta
\begin{bmatrix}
y\\
1
\end{bmatrix} 
\]
for some unit $\eta\in \O_E^\times$.

Suppose 
\[ g=
\begin{bmatrix}
a &b\\
c &d
\end{bmatrix}\in  \Gamma^\pm_0(N).
\]
Then $(ay+b)/(cy+d) = y$.  It follows that $ay+b = cy^2+dy$, and so
\[
 y^2+((d-a)/c)y -b/c=0.
\]
This must be the minimal polynomial for $y$, and since $y$ is integral, $a\equiv d \pmod c$.

Now we show $[e,ge]_\Gamma = 0$.  
It will be clearer here to write a modular symbol $[a_1,a_2]$ as a $2\times 2$ matrix, where 
the $i$-th  column is a vector in $\Q^2$ representing the point $a_i$ in projective space.
First note that
\[[e,ge]_\Gamma=[g\inv e, e]_\Gamma=-[e,g\inv e]_\Gamma.\]
Next, since $g\inv$ equals $
\pm\mat{
d &-b\\
-c &a
}$, we have 
\[
\mat{
1 &a\\
0 &c
}_\Gamma=
-\mat{
1 &d\\
0 &-c
}_\Gamma=
-\mat{
1 &-d\\
0 &c
}_\Gamma=
-\mat{
1 &d\\
0 &c
}_\Gamma, 
\]
where the last step uses $\diag(1,-1) \in \Gamma$.  But $a = d+kc$ for some integer $k$, so we can multiply the last modular symbol by
$
\mat{
1 &k\\
0 &1
}$
 without changing its value.  We obtain that
$
\mat{
1 &a\\
0 &c
}_\Gamma
$
is equal to minus itself and therefore equals $0$.
\end{proof}

\section{Application of Lenstra's theorem}\label{lenstrasection}

In this section we use a theorem of Lenstra~ \cite[p.~203]{L}  to prove a result we need in the rest of the paper.
First we recall Lenstra's theorem.  

Lenstra's notation: Fix a prime number $p$.   Let $K$ be a global field, $F$ a finite Galois extension of $K$, a subset $C\subset\Gal(F/K)$ which is a union of conjugacy classes, a finitely generated subgroup $W\subset K^\times$ of rank $r\ge1$ modulo its torsion subgroup, and an integer $k>0$ which is relatively prime to $p$.   If $t$ is an automorphism of a field $A$ and $B$ is a subfield of $A$ stable under $t$, then $t|_B$ denotes the restriction of $t$ to $B$. Let $(~,~)$ denote the Artin symbol.
 For any positive integer $a$, 
$\zeta_a$ denotes a primitive $a$-th root of unity. 

Let
$\bar K_{\fp}$ denote the ring of integers of $K$ modulo a prime $\fp$ of $K$. For any prime $\ell\ne p$, $q(\ell)$ equals the smallest power of $\ell$ not dividing $k$ 
and $L_\ell= K(\zeta_{q(\ell)}, W^{1/q(\ell)})$.
\begin{definition}
$M$ is the set of primes $\fp$ of $K$ satisfying the following conditions:
  \begin{enumerate}
  \item $(\fp, F/K)\subset C$,
  \item $\ord_{\fp} (w) = 0$ for all $w\in W$,
  \item if $\psi \colon W\to \bar K_{\fp}^\times$ is the natural map, then the index of $\psi(W)$ in $K_{\fp}^\times$ divides $k$.
  \end{enumerate}
\end{definition}

Lenstra's theorem (4.6) (slightly paraphrased): 
\begin{theorem}
Let $h$ be the product of those prime numbers $\ell\ne p$ for which $W\subset (K^{\times})^{q(\ell)}$.  Then $M$ is infinite if and only if there exists $\sigma\in\Gal(F(\zeta_h)/K)$ with
\begin{enumerate}
\item $\sigma |_F\in C$, and
\item $\sigma |_{L_\ell} \not= \id_{L_\ell}$ for every $\ell$ with $L_\ell \subset F(\zeta_h)$.
\end{enumerate}
\end{theorem}

Fix a real quadratic field $E$.  In Lenstra's notation,
set $K=E$ and 
 $k = 1$, so that 
$q(\ell)=\ell$, for all $\ell$.  Let $W$ be a subgroup of finite index in $\O_E^\times$, and assume $-1\in W$.
Let $h$ be the product of those primes $\ell$ such that $W\subset E^\ell$.

\begin{theorem}\label{lenstratheorem}
Let $c$ be an integer, $c\ne0$, $m\ge 5$  a rational prime number that
is prime to $hc$, prime to the discriminant of $E/\Q$
and such that $c\not\equiv 1 \bmod m$.
Assume GRH.   

Then there are infinitely many primes $p$ in the arithmetic progression $c+km$ such that
\begin{enumerate}
\item \label{it:len-1} $p\O_E=\pi\pi'$ splits in $E$; and
\item \label{it:len-2} the image of $W$ modulo $\pi$ is all of 
$(\O_E/\pi)^\times\simeq (\Z/p\Z)^\times$.
\end{enumerate}
\end{theorem}

\begin{proof}
 Let
\[F=E(\zeta_m),\] and for any prime $\ell$, 
\[L_\ell=E(\zeta_\ell,W^{1/\ell}).\]

 Let $n$ be a positive integer relatively prime to $m$  and such that 
$E\subset \Q(\zeta_{n})$.  Then 
$F=E(\zeta_m)\subset \Q(\zeta_n,\zeta_m)$.
There exists $\tau\in 
\Gal(\Q(\zeta_n, \zeta_m)/\Q)=
\Gal(\Q(\zeta_n)/\Q)\times \Gal(\Q(\zeta_m)/\Q)$
such that $\tau$ is the identity on $\Q(\zeta_n)$ and 
$\tau(\zeta_m)=\zeta_m^c$.
Then setting $\sigma_c=\tau |_F$, we have that 
$\sigma_c\in \Gal(F/E)$ and $\sigma_c(\zeta_m)=\zeta_m^c$. 
Set Lenstra's $C=\{\sigma_c\}$ the singleton conjugacy class in $\Gal(F/E)$.

Then~\cite[Theorem 4.6, p. 208]{L} implies the following:
\begin{lemma}\label{L} 
Assume  GRH.   Then the set of prime ideals
 $\fp$ in $\O_E$ such that 
\begin{enumerate}[label=(\alph*)]
\item \label{it:L-a} the Frobenius of a prime above $\fp$ in $F/E$ is the automorphism 
$\sigma_c$; and 

\item \label{it:L-b} the image of $W$ modulo $\fp$ is all of 
$(\O_E/\fp)^\times$
\end{enumerate}
has positive density if (and only if)
there exists $\sigma\in \Gal(F(\zeta_h)/E)$ such that
\begin{enumerate}
\item \label{it:L-1} 
 $\sigma |_F = \sigma_c$;
and
\item \label{it:L-2} 
for every rational prime $\ell$,  if $L_\ell$ is contained in $F(\zeta_h)$,
then
$\sigma |_{L_\ell}$ is not the identity permutation.
\end{enumerate}
\end{lemma}

\noindent  With this lemma in mind, we first we construct $\sigma$ that satisfies \ref{it:L-1}, and then we  verify \ref{it:L-2}.

Note that $F(\zeta_h)=E(\zeta_h,\zeta_m)\subset \Q(\zeta_n,\zeta_{6h},\zeta_m)$.  
Now
\[
\Gal( \Q(\zeta_n,\zeta_{6h},\zeta_m)/ \Q)=
\Gal(\Q(\zeta_n,\zeta_{6h})/\Q)\times \Gal(\Q(\zeta_m)/\Q).
\]
Let $\phi\in \Gal( \Q(\zeta_n,\zeta_{6h},\zeta_m)/ \Q)$ be the element  that is complex conjugation on $\Q(\zeta_n,\zeta_{6h})$ and 
$\sigma_c | \Q(\zeta_m)$ on $\Q(\zeta_m)$.  Set $\sigma=\phi |_{F(\zeta_h)}$.

Then \ref{it:L-1} is satisfied because $E$ is real, and $F$ is the compositum of $E$ and 
$\Q(\zeta_m)$.

As for \ref{it:L-2},  assume that $L_\ell \subset F(\zeta_h)$.
First suppose $\ell \nmid 6hm$.  By our supposition,  
\[E(\zeta_\ell)\subset E(\zeta_m,\zeta_h)=E(\zeta_{hm}).\]
 Consider the diagram, with the degrees of the extensions as shown:
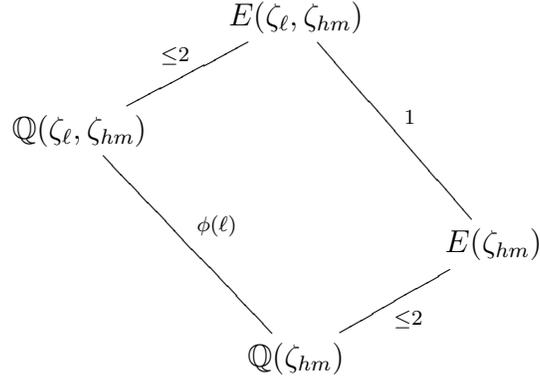
\[
\xymatrix{
&E(\zeta_{\ell},\zeta_{hm})  \ar@{-}[ddr]^1&\\
\Q(\zeta_{\ell},\zeta_{hm})  \ar@{-}[ur]^{\le2}  \ar@{-}[ddr]^{\phi(\ell)}  && \\
&&E(\zeta_{hm})  \ar@{-}[dl]^{\le2}  \\
&\Q(\zeta_{hm})&
}
\]
It follows that $\phi(\ell)\le2$ and therefore $\ell=2$ or $3$, which is contra hypothesis.

Therefore $\ell \mid 6hm$.  If $\ell=m$, then  
$\sigma |_{L_\ell}$ is not the identity permutation because it raises $\zeta_m$ to the $c$ power.  If $\ell \mid 3h$, then 
$\sigma |_{L_\ell}$ is not the identity permutation because it acts as complex conjugation on $\zeta_\ell$. 
If $\ell = 2$, then 
$\sigma |_{L_\ell}$ is not the identity permutation because it acts as complex conjugation on $\sqrt{-1}\in L_\ell$.

By Lemma~\ref{L}, we now know there are infinitely many ideals $\fp$ satisfying \ref{it:L-a} and \ref{it:L-b}.  We claim that all but finitely many such $\fp$ are split in $E$.  Suppose $\fp$ is inert.  Then $\O_E/\fp=\F_{p^2}$ where $\fp$ lies over the rational prime $p$.
Now (b) says that the image of $W$ modulo $\fp$ is all of 
$(\O_E/\fp)^\times$.  But $W/\pm1$ is cyclic, so this image 
is generated by $\pm 1$ and some
$x$ which is the reduction of some power of $\epsilon$ modulo $\fp$.  Globally, the Galois conjugate of $\epsilon$ is $\epsilon\inv$, so the $\Gal(\F_{p^2}/\F_p)$-conjugate of $x$ is $x\inv$.  That means that if $x^k\in \F_p$ then 
$x^k=x^{-k}$ and therefore $x^k=\pm 1$.  
So $\set{ \text{$\pm$ the powers of $x$}}$  cannot equal all of $(\O_E/\fp)^\times$ if $p>3$.

Now we have infinitely primes $p$ satisfying \ref{it:len-1} and \ref{it:len-2} of Theorem~\ref{lenstratheorem} and such that 
\ref{it:L-a} and \ref{it:L-b} of Lemma~\ref{L} holds for $\fp=\pi$.
  We claim that if such a $p$ is sufficiently large, then it is in the arithmetic progression $c+km$.  Indeed,
\ref{it:L-a} says that the Frobenius of a prime 
$\fP$ above $\fp$ in $E(\zeta_m)/E$ is the automorphism $\sigma_c$.  That means, since $p$ is split, that raising to the $p$-th power in 
$\Gal(\O_ {E(\zeta_m)}/\fP/\O_E/\fp)$ is induced by $\sigma_c$.  Therefore
\[
\zeta_m^c\equiv\zeta_m^p\bmod \fP.
\]
Therefore if the images of the $m$-th roots of unity in $\O_{E(\zeta_m)}/\fP$
are pairwise distinct, then $p$ is congruent to $c$ modulo $m$. But this will be true if $p$ is sufficiently large.  (In fact, if $S$ is any finite subset of 
$\O_{E(\zeta_m)}$, the set of prime ideals in $\O_{E(\zeta_m)}$ dividing any
member of the set $\set{s_1-s_2 \given s_1\ne s_2,\ s_1,s_2\in S}$ is finite.)
\end{proof}

\section{The images of \texorpdfstring{$\psi_{\Gamma_1(N)^\pm}\  {\text{and}} \ \psi_{\Gamma_1(N)}$}{psi-Gamma-1(N)-pm, \ and psi-{Gamma-1(N)}}}  \label{im1}

Fix a real quadratic field $E$, and fix a coefficient ring $R$.
For the first few paragraphs of this section, let $\Gamma$ be any subgroup of $\GL_2(\Q)$.
Recall from Definition~\ref{N3} that 
$K(\Gamma)$ is the group generated by all the unital 
and triangular matrices in $\Gamma$.  

Let $\set{e_1,e_2}$ be the standard basis of 
$\Z^2$ and recall that $e$ is the image of $e_1$ in $\PP^1(\Q)$.  Let 
$f$ denote the image of $e_2$ in $\PP^1(\Q)$.

\begin{lemma} \label{lem:tri-image}
If $u\in\Gamma$ is triangular, then $[e,ue]_\Gamma=0$, and hence it is in the image of $\psi_\Gamma$.
\end{lemma}

\begin{proof}
If $u\in\Gamma$ is upper triangular, then 
$[e,ue]_\Gamma=
[e,e]_\Gamma=0$.

If $u\in\Gamma$ is lower triangular, then 
\[
  [e,ue]_\Gamma =
  [e,f]_\Gamma+[f,ue]_\Gamma =
  [e, f]_\Gamma-[ue,uf]_\Gamma =
  0.  \]
\end{proof}

\begin{lemma}\label{Kgen}  
The image of $\psi_\Gamma$ is the $R$-span of $[e,\gamma e]_\Gamma$ where $\gamma$ runs over $K(\Gamma)$.
\end{lemma}

\begin{proof}
By Corollary~\ref{cor:im_psi}, Lemma~\ref{S}, and Lemma~\ref{lem:tri-image}, 
if $k\in K(\Gamma)$, then $[e,ke]_\Gamma\in\image\psi_\Gamma$.  
On the other hand, by definition, the image of $\psi_\Gamma$ is the $R$-span of 
$[e,ge]_\Gamma$ where $g$ runs over just the unital matrices in $\Gamma$, which are all in  $K(\Gamma)$.
\end{proof}

Compare Lemma~\ref{Kgen} to Corollary~\ref{widetilde}.  The set of symbols involved is getting larger, but their span remains the same.

\begin{lemma}\label{KG}  
Let $\Gamma$ be a subgroup of $\GL_2(\Q)$, $R$ any ring.  
Suppose $K(\Gamma)=\Gamma$. Then the
 image of $\psi_{\Gamma}$ is all of
 $H_0^{\cusp}(\Gamma, \St(\Q^2;R))$.
\end{lemma}

\begin{proof}
By Lemma~\ref{hcusp} \eqref{eq:hcusp-a}, $H_0^{\cusp}(\Gamma, \St(\Q^2;R))$ is the $R$-span of all 
 $[e,\gamma e]_\Gamma$ as $\gamma$ runs over $\Gamma$.   
 By Lemma~\ref{Kgen}, the image of $\psi_\Gamma$ is the $R$-span of all 
 $[e,ke]_\Gamma$ as $k$ runs over $K(\Gamma)$.
Since $K(\Gamma)=\Gamma$,  these two $R$-spans are the same.
 
\end{proof}

\begin{theorem}\label{maintheorem}
Let $R$ be any coefficient ring and $N\ge1$.   Assume GRH.  Let 
$\Gamma=\Gamma_1(N)^\pm$ or $\Gamma_1(N)$
Then:
\begin{enumerate}
\item \label{it:main-a} $K(\Gamma)=\Gamma$.
\item \label{it:main-b}  The image of $\psi_{\Gamma}$ 
is all of $H_0^{\cusp}(\Gamma, \St(\Q^2;R))$.
\end{enumerate}
\end{theorem}

\begin{proof}
By Lemma~\ref{KG}, assertion~\ref{it:main-a} implies assertion~\ref{it:main-b}
It remains to prove \ref{it:main-a}.

First let $\Gamma=\Gamma_1(N)^\pm$.
Given
\[
\delta =
\begin{bmatrix}
a&b\\
Nc&d
\end{bmatrix}\in\Gamma
\]
we must show that $\delta\in K(\Gamma)$.  
It suffices to show that there exists $\gamma\in K(\Gamma)$ such that 
$\gamma e_1=\delta e_1$, because then $z=\gamma\inv \delta\in\Gamma$ and stabilizes $e_1$ and so is triangular. For then $z\in K(\Gamma)$ and 
 $\delta=\gamma z$ is in  $K(\Gamma)$ also.
 
 So we have to prove the following lemma:

\begin{lemma}\label{mainlemma}
Let $N\ge1$, and let $\Gamma=\Gamma_1(N)^\pm$.  Let $a$ and $c$ be relatively prime integers such that $a\equiv\pm1\bmod N$.
 Assume GRH.  Then there exists 
$\gamma\in K(\Gamma)$ such that 
\[
\gamma e_1=
\begin{bmatrix}
a\\
Nc
\end{bmatrix}.
\]
\end{lemma}

Proof of the lemma:
First we assume that $c\ne1$.

We may pre- and post-multiply $\gamma$ by elements in $K(\Gamma)$ as needed.  All unipotent triangular matrices that are in $\Gamma$ are also in $K(\Gamma)$.
We use Dirichlet's theorem on primes in an arithmetic progression
and Theorem~\ref{lenstratheorem}.  
Let
\[
w=
\begin{bmatrix}
a\\
Nc
\end{bmatrix}.
\]
If $u\in\Gamma$
is an upper triangular unipotent matrix with $t\in\Z$ in the upper right hand corner then 
\[
uw=
\begin{bmatrix}
a+tNc\\
Nc
\end{bmatrix}.
\]
Therefore, there exists an upper triangular unipotent matrix $u$,  such that 
\[
uw=\begin{bmatrix}
m\\
Nc
\end{bmatrix}
\]
where $m$ is a prime number, prime to $6Nc$, prime to the discriminant of $E$, such that $U_{4N}^\pm$ is not contained in $E^m$, and such that $c\not\equiv 1 \bmod m$.  (Just take $m$ sufficiently large.)  

 By Theorem~\ref{lenstratheorem},  there are infinitely many primes $p$ in the arithmetic progression $c+km$ such that

 \begin{enumerate}
 \item \label{it:prog-1} $p\O_E=\pi\pi'$ splits in $E$; 
\item  \label{it:prog-2} The image of $U_{4N}^\pm$ modulo $\pi$ is all of 
$(\O_E/\pi)^\times\simeq (\Z/p\Z)^\times$. 
 \end{enumerate}

Now  
if $v\in\Gamma$
is a lower triangular unipotent matrix with $Nk$ in the lower right hand corner then 
\[
vuw=
\begin{bmatrix}
m\\
Nc+Nkm
\end{bmatrix}.
\]
Choose $k$ so that $c+km$ is an odd prime $p$ satisfying \ref{it:prog-1} and \ref{it:prog-2} and $p$ does not divide $N$.
 
Now 
\[
vuw=
\begin{bmatrix}
m\\
Np
\end{bmatrix}.
\]

By Theorem~\ref{main2}, there exists a matrix $M_N(m,p)$ in $K(\Gamma)$ such that 
$M_N(m,p)e_1=vuw$.

Finally, we have to take care of the case where $c=1$.   Given
\[
w=
\begin{bmatrix}
a\\
N
\end{bmatrix},
\]
$a=\pm1+tN$.  Then premultiplying $w$ by the upper triangular unipotent matrix $u$ with $-t$ in the upper right hand corner, we have
\[
uw=
\begin{bmatrix}
\pm1\\
N
\end{bmatrix}.
\]
Then 
\[
v :=
\begin{bmatrix}
\pm1&0\\
N&\pm1
\end{bmatrix}\in K(\Gamma)
\]
and $u\in K(\Gamma)$,
so $u\inv v\in K(\Gamma)$ and
$u\inv ve_1 =w$.

Now let $\Gamma=\Gamma_1(N)$ 
and $\Gamma^\pm=\Gamma_1(N)^\pm$.
Part~\ref{it:main-a} asserts that $K(\Gamma^\pm)=\Gamma^\pm$.  This means that any $\gamma\in\Gamma^\pm$ is a product of unital and triangular elements of 
$\Gamma^\pm$.  However, inspection of its proof and the proof of Theorem~\ref{main2} shows that in fact we only need one unital element and that of determinant $1$:
  any $\gamma\in\Gamma^\pm$ can be written as 
\[
\gamma=  u_1vu_2Mu_3,
\] 
where
$u_1,u_2,u_3$ are upper triangular elements in $\Gamma^\pm$, $v$ is a lower
triangular element in $\Gamma^\pm$ and $M$ is a unital element in $\Gamma^\pm$ with 
$\det(M)=1$.

Let $S$ denote the 4-group $\{I,-I,s,-s\}$ where $s=\diag(1,-1)$.
Note that conjugation by $S$ stabilizes $\Gamma$, the set of upper triangular matrices, and the set of lower triangular matrices, and therefore also $K(\Gamma)$.  For any $x\in \Gamma^\pm$, there exists an element $t\in S$ such that $xt\in\Gamma$.
 
 Now suppose $\gamma\in\Gamma$.   We may write
\[
\gamma=u_4wu_5(Ms_1) u_6s_2,
\] 
where $s_1,s_2\in S$, $u_4,u_5,u_6$ are upper triangular elements in $\Gamma$, $w$ is a lower
triangular element in $\Gamma$ and $Ms_1$ is  in $\Gamma$.  From this it follows that $s_2\in\Gamma$ and hence $s_2=\pm I$ (in fact if $N>2$, $s_2=I$). Therefore 
\[
\pm\gamma=u_4wu_5(Ms_1) u_6.
\] 
Since $M$ has determinant $1$, as do $\gamma, u_4, w, u_5$ and $ u_6$, it follows that 
$\det(s_1)=1$ and therefore $s_1=\pm I$.
Since $-I$ is $\beta$-unital for any $\beta$, it follows that $M'=Ms_3$ is a unital element of 
$\Gamma$.  We obtain
\[
\pm\gamma=u_4wu_5M' u_6\in K(\Gamma).
\]
Now assume that $N>2$.  Then the left hand side is actually $\gamma$,
and therefore $K(\Gamma)=\Gamma$.  On the other hand, if $N=1$ or $2$, 
then $-I\in\Gamma$ and is triangular, so again $\gamma\in K(\Gamma)$ and therefore $K(\Gamma)=\Gamma$.
\end{proof}

\section{The cokernel of \texorpdfstring{$\psi_\Gamma$}{psi-Gamma} for other congruence subgroups
\texorpdfstring{$\Gamma\subset \GL_2(\Q)$}{Gamma subset GL2(Q)}}  \label{imgeneral}

Fix a ring $R$.  Let us view $\psi_\Gamma$ as a homomorphism with target
$H^\cusp(\Gamma,\St(\Q^2;R))$, so we may speak of its cokernel in this regard.
 In this section, we prove that the cokernel of $\psi_\Gamma$ is a finitely generated torsion module over $R$, for all congruence subgroups
$\Gamma\subset \GL_2(\Q)$.  

To start with, let  $\Gamma\subset \GL_2(\Q)$ be any subgroup.  Later in this section will assume that $\Gamma$ is a congruence subgroup.  This means that there exists an integer $N\ge1$ such that $\Gamma$ contains $\Gamma(N)$ with finite index, where $\Gamma(N)=\set{g\in\SL_2(\Z)\given g \equiv I \bmod N}$.
 
 First we prove some lemmas. 

\begin{lemma}\label{mglemma1}
Up to isomorphism, the cokernel of $\psi_\Gamma$ doesn't change up to isomorphism if 
$\Gamma$ is conjugated by a matrix $A\in\GL_2(\Q)$.
\end{lemma}

\begin{proof}

The short exact sequence~(\ref{seq}) is equivariant for $\GL_2(\Q)$.
The long exact sequence of group homology, including the connecting homomorphisms, is functorial.  Therefore, we obtain a commutative diagram:
\[
\xymatrix{
H_1(\Gamma,C)\ar[d]^-{f_1}\ar[r]^-{\psi_{\Gamma}}
&H_0^{\cusp}(\Gamma, \St(\Q^2;R))\ar[d]^-{f_0}
\\
H_1(A\Gamma A\inv, C)\ar[r]^-{\psi_{A\Gamma A\inv}}
&H_0^{\cusp}(A\Gamma A\inv, \St(\Q^2;R))
}
\]
where $f_1$ and $f_0$ are the maps induced by conjugation by $A$ on the group and multiplication by $A$ on the coefficients.  Since $f_1$ and $f_0$ are isomorphisms, the lemma follows.
\end{proof} 

\begin{lemma}\label{mglemma5}
Let  $\Gamma$ be a subgroup of 
$\GL_2(\Q)$.  Suppose  $K(\Gamma)$ has finite index in $\Gamma$.
Then $H_0^{\cusp}(\Gamma, \St(\Q^2;R))$ modulo the image of
$\psi_\Gamma$ is a finitely-generated torsion $R$-module.
\end{lemma}

\begin{proof}
In the notation of Theorem~\ref{Psi*}, set $K=K(\Gamma)$ and $L=\Gamma$.  Then $X$ is equal to $H_0^{\cusp}(\Gamma, \St(\Q^2;R))$ modulo the image of
$\psi_\Gamma$.  Therefore, it suffices to prove that $Q^{ab}$ is a finite group.
But $Q=L/\widehat K_L$ and since $K$ has finite index in $L$, so does $\widehat  K_L$.
\end{proof}

\begin{theorem}\label{maingeneral}
Let  $\Gamma$ be a subgroup of 
$\GL_2(\Q)$ that contains $\Gamma(M)$ with finite index for some $M$.  Assume GRH.  Then $H_0^{\cusp}(\Gamma, \St(\Q^2;R))$ modulo the image of
$\psi_\Gamma$ is a finitely-generated torsion $R$-module. 
\end{theorem}

\begin{proof}
Conjugating by 
 $A=\diag(1,M)$, we see that for $N=M^2$,  
 $\Gamma_1(N)\subset A\Gamma A\inv$.
So by Lemma~\ref{mglemma1},
we may replace $\Gamma$ by $A\Gamma A\inv$;
and without loss of generality, for some $N$, we assume that 
$\Gamma_1(N)\subset \Gamma$  with finite index.

Then by Theorem~\ref{maintheorem} we have
\[
\Gamma_1(N)=K(\Gamma_1(N))\subset K(\Gamma) \subset \Gamma
\]
from which it follows that $K(\Gamma)$ has finite index in $\Gamma$.
We are finished by 
Lemma~\ref{mglemma5}.
\end{proof}

\section{The image of \texorpdfstring{$\psi_\Gamma$}{psiGamma} for \texorpdfstring{$\Gamma_0(N)^\pm$}{Gamma0(N)+-} and for \texorpdfstring{$\Gamma_0(N)$}{Gamma0(N)}
}
\label{im0}

 Let $\Gamma$ be either $\Gamma_0(N)^\pm$ or $\Gamma_0(N)$.  
Set $\Gamma_1=\Gamma_1(N)^\pm$ if $\Gamma=\Gamma_0(N)^\pm$ and 
$\Gamma_1=\Gamma_1(N)$ if $\Gamma=\Gamma_0(N)$.  
 In the first case, the quotient $\Gamma/\Gamma_1$ 
 is isomorphic to $(\Z/N\Z)^\times/\set{\pm1}$ and in the 
second case, the quotient is isomorphic to $(\Z/N\Z)^\times$.

\begin{lemma}\label{root}
Let $\eta\in \O_E^\times$ and $f_\eta(x)=x^2-tx+n\in\Z[x]$ its characteristic polynomial.  
Then $a\in(\Z/N\Z)^\times$ is a root of $f_\eta$ modulo $N$ if and only if
\[
a+na\inv\equiv t\bmod N.
\]
\end{lemma}

\begin{remark}
Let prime denote Galois conjugate.  Then  $t=\eta+\eta'$ is the trace of $\eta$ and $n=\eta\eta'=\pm1$ is its norm.
\end{remark}

\begin{proof}
Suppose $a$ is a root of $f_\eta$ modulo $N$.  Then $a$ is invertible mod $N$ because the product of the two roots is $\pm1$.  So the other root is $na\inv$
and $a+na\inv\equiv t\bmod N$.  Conversely, if $a+na\inv\equiv t\bmod N$, then $f_\eta(x)$ modulo $N$ is $x^2-(a+na\inv)x+n$ and $a$ is a root.
\end{proof}

\begin{lemma}\label{gamma0}
Let $\eta\in \O_E^\times$ and let $f_\eta(x)=x^2-tx+n\in\Z$ be its characteristic polynomial. 
Suppose $a\in(\Z/N\Z)^\times$ is a root of $f_\eta$ modulo $N$.
Then there exists a unital $\gamma\in\GL_2(\Z)$ such that the determinant 
of $\gamma$ equals the norm of $\eta$ and 
\[
\gamma\equiv
\begin{bmatrix}
a&*\\
0&na\inv
\end{bmatrix} \bmod N.
\]

Conversely, if $\gamma\in\Gamma_0(N)^\pm$ is unital then its upper left hand corner modulo $N$ is a root of $f_\eta$ modulo $N$ for some 
$\eta\in \O_E^\times$ whose norm equals the determinant of $\gamma$.
\end{lemma}

\begin{proof} If $\eta=\pm1$, take $\gamma=\pm I$.  So now assume that
$\eta\ne\pm1$.
Lift $\begin{bmatrix}
a&*\\
0&na\inv
\end{bmatrix}$
to a matrix
$
g=\begin{bmatrix}
A&B\\
C&D
\end{bmatrix}\in \GL_2(\Z)
$ that is congruent to it modulo $N$.

Then $AD-BC\equiv n$, $A+D\equiv a+na\inv$, and $C\equiv 0 \bmod N$.
So the characteristic polynomial of $g$ is congruent to $f_\eta$.  We need to modify $g$ so that its characteristic polynomial equals $f_\eta$ on the nose.

Write $C=cN$.  Conjugate $g$ by $d_c=\diag(c,c\inv)\in \GL_2(\Q)$.  We obtain
\[
g_1=d_cgd_c\inv=
\begin{bmatrix}
A&cB\\
N&D
\end{bmatrix}=
\begin{bmatrix}
A&cB\\
N&t-A-rN
\end{bmatrix}\in \GL_2(\Z),
\]
for some $r\in\Z$.
Post multiply $g_1$ by the upper triangular unipotent matrix $u_r$ with $r$ in the upper right corner, and call the result $\gamma$:
\[
\gamma=
g_1u_r=
\begin{bmatrix}
A&cB\\
N&t-A-rN
\end{bmatrix}
\begin{bmatrix}
1&r\\
0&1
\end{bmatrix}=
\begin{bmatrix}
A&B_1\\
N&t-A
\end{bmatrix}\in \GL_2(\Z).
\]
Note that the trace of $\gamma$ is $t$ and the determinant of $\gamma$ is $\pm1$ and congruent to $n$ modulo $N$, and therefore $\det\gamma=n$ if $N\ge3$.  If $N=1$ or $2$, and if $\det\gamma=-n$, go back to the beginning and redefine $g$ by replacing $B,D$ with $-B,-D$.  That has the effect of changing the sign of the determinant.  So for any $N$, we now have $\gamma$ whose trace is $t$ and whose determinant is $n$.

Therefore the characteristic polynomial of $\gamma$ is $f_\eta$, and
\[
\gamma\equiv
\begin{bmatrix}
a&*\\
0&na\inv
\end{bmatrix}\bmod N.
\]
It remains to show that $\gamma$ is unital.

The eigenvalues of $\gamma$ are $\eta$ and $\eta'\in E$ which are not equal.  Therefore, $\gamma$ is diagonalizable over $E$, with an eigenvector
$\begin{bmatrix}
\beta\\
1
\end{bmatrix}$
for some $\beta\in E\setminus\Q$.  So
\[\gamma\begin{bmatrix}
\beta\\
1
\end{bmatrix}=
\eta\begin{bmatrix}
\beta\\
1
\end{bmatrix},\]
and $\gamma$ is $\beta$-unital.

Conversely, suppose 
\[
\gamma=
\begin{bmatrix}
A&B\\
cN&D
\end{bmatrix}
\in\GL_2(\Z)
\] 
is unital.
Then there exists 
$\eta\in \O_E^\times$ and $\beta\in E \setminus Q$ such that
\[\gamma\begin{bmatrix}
\beta\\
1
\end{bmatrix}=
\eta\begin{bmatrix}
\beta\\
1
\end{bmatrix}.
\]
The characteristic polynomial of $\gamma$ is the same as 
the characteristic polynomial of $\eta$, namely $f_\eta$.   In particular the norm of $\eta$ equals the determinant of $\gamma$.
Also the characteristic polynomial of $\gamma$  is
congruent modulo $N$ to $(x-A)(x-D)$ and has $A\bmod N$ as a root.
\end{proof}

\begin{definition}\label{AE}
Let $E$ be a real quadratic field and $N$ a positive integer. 
 
Define $A_E(N)$ to be the subgroup of $(\Z/N\Z)^\times/\set{\pm1}$ generated by the images of those $a\in (\Z/N\Z)^\times$ such that $a$ is a root of $f_\eta$ for 
some $\eta\in \O_E^\times$.

Define $A_E(N)^*$ to be the subgroup of $(\Z/N\Z)^\times$ generated by those $a\in (\Z/N\Z)^\times$ such that $a$ is a root of $f_\eta$ for some $\eta\in \O_E^\times$ of norm 1.
\end{definition}

\begin{remark}
Given $N$ and $E$, it is a finite computation to determine $A_E(N)$.  This is because the powers of $\overline\epsilon:=$ the image of $\epsilon$ in $\O_E/N\O_E$ constitute a finite set $Q$, which can be found be seeing when 
$\overline\epsilon, \overline\epsilon^2,\dots$ begins to repeat.  Then by 
Lemma~\ref{root}, $a$ is a root of
$f_\eta$ for some $\eta\in \O_E^\times$ if and only if 
$
a+na\inv= t
$,
where $t=\overline\epsilon^k+\overline{\epsilon'}^k$ and $n=\overline\epsilon^k\overline{\epsilon'}^k$ for some $\overline\epsilon^k\in Q$.
However, we do not know any simple formula that describes $A_E(N)$ given $N$ and $E$.  In the range of our computations ($N \leq 1000, \Delta \leq 50$), there are $6846$ different groups that arise for $A_E(N)$ and $6419$ for $A_E(N)^*$.  A sample ($N \leq 20$) is given in Tables~\ref{tab:ane} and \ref{tab:SL-ane}.

We do not know how to predict $A_E(N)$ from $A_E(N)^*$ or vice versa, in general.  Of course, $A_E(N)^*$ equals $A_E(N)/\set{\pm1}$ if $E$ has no unit of norm $-1$, but otherwise, it may be equal or it may have index two.
\end{remark}

Let $\pi\colon \Gamma_0(N)^\pm\to (\Z/N\Z)^\times/\set{\pm1}$ be the surjective homomorphism that sends a matrix to its upper left hand corner modulo $N$ and then modulo $\pm1$. 
Similarly, let $\pi^*\colon \Gamma_0(N)\to (\Z/N\Z)^\times$ be the surjective homomorphism that sends a matrix to its upper left hand corner modulo $N$.  

\begin{theorem} \label{00}
  Assume GRH.  Let $E$ be a real quadratic field.
Then $K(\Gamma_0(N)^\pm,E)=\pi\inv({A_E(N)})$ 
and
   $K(\Gamma_0(N),E)=(\pi^*)\inv(A_E(N)^*)$. 
\end{theorem}

The following corollary follows immediately from the theorem and Lemma~\ref{Kgen} :
\begin{corollary}  Assume GRH.  Let $E$ and $F$ be real quadratic fields.
  If $A_N(E)=A_N(F)$, then the cokernels of 
    $\psi_{\Gamma_0(N)^\pm,E}$ and
  $\psi_{\Gamma_0(N)^\pm,F}$ are the same. 
  If $A_N(E)^* = A_N(F)^*$ 
  then the cokernels of $\psi_{\Gamma_0(N),E}$ 
  and $\psi_{\Gamma_0(N),F}$
  are the same. 
\end{corollary}

Now we prove the theorem:

\begin{proof}
 The group $K(\Gamma_0(N)^\pm)$ is generated by
$K(\Gamma_1(N)^\pm)$ and the unital elements in $\Gamma_0(N)^\pm$ (since the triangular elements of $\Gamma_0(N)^\pm$ are already in $\Gamma_1(N)^\pm$.) 
Let prime denote reduction modulo $ \{\pm I\}$.
By definition,  
$\pi(K(\Gamma_1(N)^\pm))\subset \{\pm 1\}'\subset A_E(N)$. 
By Lemma~\ref{gamma0}, if $\gamma\in \Gamma_0(N)^\pm$ is unital, then $\pi(\gamma)\in
A_E(N)$.  
So $\pi(K(\Gamma_0(N)^\pm))\subset A_E(N)$.
Conversely,
suppose $\gamma\in \Gamma_0(N)^\pm$ and $\pi(\gamma)\in A_E(N)$.  By Lemma~\ref{gamma0}, there exists a unital 
$\delta\in \Gamma_0(N)^\pm$
such that $\gamma\delta\inv$ has its upper left corner congruent to 1 modulo $N$.  So $\gamma\delta\inv\in\Gamma_1(N)^\pm$ and hence in 
$K(\Gamma_1(N)^\pm)$ by Theorem~\ref{maintheorem}.
So both $\delta$ and $\gamma\delta\inv$ are in $K(\Gamma_0(N)^\pm)$,
and hence so is $\gamma$.

For the second assertion, notice that the group $K(\Gamma_0(N))$ is generated by
$K(\Gamma_1(N))$ and the unital elements in $\Gamma_0(N)$,
including $\pm I$ (since the triangular elements of $\Gamma_0(N)$ are  in 
$\pm\Gamma_1(N)$.) 
By definition,  
$\pi^*(K(\Gamma_1(N)))\subset \{1\}\subset A_E^*(N)$. 
By Lemma~\ref{gamma0}, if $\gamma\in \Gamma_0(N)$ is unital, then 
$\pi^*(\gamma)\in
A_E(N)^*$.  
So $\pi^*(K(\Gamma_0(N))\subset A_E(N)^*$.
Conversely,
suppose $\gamma\in \Gamma_0(N)$ and $\pi^*(\gamma)\in A_E^*(N)$.  By Lemma~\ref{gamma0}, there exists a unital 
$\delta\in \Gamma_0(N)$
such that $\gamma\delta\inv$ has its upper left corner congruent to 1 modulo $N$.  So $\gamma\delta\inv\in\Gamma_1(N)$ and hence in 
$K(\Gamma_1(N))$ by Theorem~\ref{maintheorem}.
So both $\delta$ and $\gamma\delta\inv$ are in $K(\Gamma_0(N))$,
and hence so is $\gamma$.
\end{proof}

\section{Conjectures}\label{conj}

Our computational results give us further confidence to make the following conjecture, independent of any assumption of GRH.  Because in our computations we used $R=\Z$, we will assume that $R=\Z$.  

\begin{cnj}\label{Q}
For any level $N$ and real quadratic field $E$:
  \begin{enumerate}
  
\item   If $\Gamma = \Gamma_1(N)^\pm$ or $\Gamma_1(N)$, then
$\psi_{\Gamma,E}$ is surjective onto 
$H_0^{\cusp}(\Gamma, \St(\Q^2;\Z))$. 

 \item  The cokernel of $\psi_{\Gamma_0(N)^\pm,E}$ is isomorphic to a quotient of 
$((\Z/N\Z)^\times /\set{\pm1})/A_E(N)$. 

\item The cokernel of $\psi_{\Gamma_0(N),E}$ is isomorphic to a quotient of 
$(\Z/N\Z)^\times/A_E(N)^*$.
  \end{enumerate} 
\end{cnj}

Our computations strongly support this conjecture.  
For $N\le1000$ and $E=\Q(\sqrt\Delta)$ for $\Delta\le50$ we checked the last two statements.
We did not compute with the $\Gamma_1$ groups because of their larger index in 
$\GL_2(\Z)$.   But the fact that the first statement enters crucially into the proofs of the latter two statements when we assume GRH in Theorem~\ref{grh} below, 
suggests that it too is true.

The question arises as to whether the surjections mentioned in (ii) and (iii) might be isomorphisms.  This is not the case.  A glance at the data in the tables in Section~\ref{results} shows that 
the cardinality of the cokernel of $\psi_{\Gamma_0(N)^\pm}$ may  be less than
$\phi(N)/(2|A_E(N)|)$ and the cardinality of the 
cokernel of $\psi_{\Gamma_0(N)}$ may be less than
$\phi(N)/|A_E(N)^*|$.

For $N$ and $E$ where do not get equality, we say ``shrinkage has occurred.''  In our data the shrinkage is always by a factor of $2$ or $4$.  We do not have an explanation for shrinkage, but it may be connected with the order of the subgroup of diagonal matrices in 
$\Gamma_0(N)^\pm$.

\begin{theorem}\label{grh}
Assume GRH.  Then the assertions in Conjecture~\ref{Q} are true.
\end{theorem}

\begin{proof}
Assertion (i) is Theorem~\ref{maintheorem}. 
For assertion (ii)  start with
Theorem~\ref{maintheorem} (i).  
Then $K(\Gamma_0(N)^\pm)\supset K(\Gamma_1(N)^\pm)=\Gamma_1(N)^\pm$.
In the notation of Theorem~\ref{Psi*}, let $R=\Z$, 
$K=K(\Gamma_0(N)^\pm)$ and $L=\Gamma_0(N)^\pm$.
Then $Q$ is a quotient of the abelian group
$\Gamma_0(N)^\pm/K(\Gamma_0(N)^\pm)$.
Now $K(\Gamma_0(N)^\pm)$ contains $\Gamma_1(N)^\pm$ and also
by Theorem~\ref{00}, 
for each element $\alpha\in A_E(N)$, it contains a unital element whose upper left corner is congruent to 
$\pm\alpha$ modulo $N$.  Therefore $Q$ is a quotient of $((\Z/N\Z)^\times /\set{\pm1})/A_E(N)$ and we are finished by Theorem~\ref{Psi*}.
The proof of assertion (iii) is similar.

\end{proof}

\section{Complex cuspidal Steinberg homology and toral periods of modular forms}\label{complex}
In this section we investigate the relationship between 
$H^{\cusp}_0(\Gamma, \St(\Q^2;\C))$  and  toral periods of 
holomorphic cusp forms of weight $2$.  Based on Theorem~\ref{grh}, we prove (assuming GRH) that the toral homology classes for a fixed real quadratic field $E$ span the homology of the compact modular curve, when 
$\Gamma=\Gamma_1(N)^\pm$, $\Gamma_1(N)$, $\Gamma_0(N)^\pm$, or 
$\Gamma_0(N)$.

Fix the coefficient ring $R$ to be the complex numbers $\C$.   
Let $\Gamma$ be a congruence subgroup of $\GL_2(\Z)$.  We assume that 
$\Gamma$ either is in $\SL_2(\Z)$ or else contains $J=\diag(1,-1)$.

Let $Y$ denote the upper half plane and let $Y(\Gamma)=\bar Y/\Gamma$, 
where $\bar Y$ is the Borel-Serre compactification of $Y$.  The quotient $\bar Y(\Gamma)$ is a compact orbifold with boundary, where the stabilizers of singular points are finite groups.  It is orientable if $\Gamma\subset\SL_2(\Z)$.  The boundary of $\bar Y$ consists of a line at each cusp $v$.  Fix a base point $z_v$ on the line $L_v$ at the cusp $v$ and let 
$\hat z_v$ denote its image in $Y(\Gamma)$.

\begin{definition}\label{toral0}
Define the map $b \colon H_0(\Gamma,\St(\Q^2;\C))=\St(\Q^2;\C)_\Gamma \to
H_1(Y(\Gamma), \partial Y(\Gamma),\C)$ by sending $[v,w]_\Gamma$ to the fundamental class of a path from $\hat z_v$ to $\hat z_w$ modulo $\Gamma$ .  
\end{definition}
It does not matter which path is chosen, since $\bar Y$ is contractible.  It is easy to see from Theorem~\ref{thm:mod-props} that the map $b$ is well-defined.

\begin{lemma}\label{b}
The map $b$ is an isomorphism.
\end{lemma}

\begin{proof} 
First suppose that 
$\Gamma\subset\SL_2(\Z)$.
By Lefschetz duality, $H_1(Y(\Gamma), \partial Y(\Gamma),\C)$ is isomorphic to $H^1(Y(\Gamma), \C)$.
Because the stabilizers of the action of $\Gamma$ on $Y(\Gamma)$ are finite groups, 
$H^1(\Gamma, \C)$ is isomorphic to $H^1(Y(\Gamma),\C)$.
By Theorem~\ref{BS},  $H_0(\Gamma, \St(\Q^2;R))$ and  $H^1(\Gamma, \C)$ are vector spaces of the same (finite) dimension.  So the source and target of $b$ are finite dimensional vector spaces of the same dimension, and it suffices to show that $b$ is surjective.

$H_1(Y(\Gamma), \partial Y(\Gamma),\C)$ is generated by the fundamental classes of paths in $Y(\Gamma)$ which either start at one cusp and end at another cusp, or are closed.   It suffices to show that these are in the image of $b$.
Choose a path $\pi$ in $Y(\Gamma)$ whose fundamental class is
$\xi\in H_1(Y(\Gamma), \partial Y(\Gamma),\C)$.  
If $\pi$ is a closed path (which may not meet
$ \partial Y(\Gamma)$), choose a point $p$ on $\pi$ and a cuspidal point 
$\hat z_x$, and add to $\pi$ a path from $\hat z_x$ to $p$ at the start and the same path backwards from $p$ to $\hat z_x$ at the end.  If $\pi$ is not a closed path, it must begin at a point which is the image modulo $\Gamma$ of some point on the line $L_v$ at some cusp $v$ and end at another such point at the same or another cusp. Without changing the homology class we may assume that these points are among our chosen base points.
So without loss of generality, we may assume that the initial and final points of $\pi$ are $\hat z_v$ and $\hat z_w$ for some cusps $v$ and $w$.
Lift $\pi$ back to a path on $\overline Y$ from $z_v$ to $z_w$.  We see that
$b([v,w]_\Gamma)=\xi$.  So $b$ is surjective and therefore bijective.

Now suppose 
$J=\diag(1,-1)\in\Gamma$, let $\widetilde\Gamma=\Gamma\cap\SL_2(\Z)$.  The map $b$ for $\widetilde\Gamma$ is an isomorphism and it is $J$-equivariant.  Here $J$ acts on 
$[v,w]_\Gamma$ by sending it to $[Jv,Jw]_\Gamma$ and on $Y$ by sending 
$z=x+iy$ to $-\overline z=-x+iy$.
We finish by noting that  $H_0(\Gamma,\St(\Q^2;\C))=H_0(\widetilde\Gamma,\St(\Q^2;\C))_{\Gamma/\widetilde\Gamma}$ and 
$H_1(Y(\Gamma), \partial Y(\Gamma),\C)=
H_1(Y(\widetilde\Gamma), \partial Y(\widetilde\Gamma),\C)_{\Gamma/\widetilde\Gamma}$.
\end{proof}

For $\Gamma\subset \SL_2(\Z)$, let $S_2(\Gamma)$ (resp. $\overline{S_2(\Gamma)}$) be the space of holomorphic 
(resp. antiholomorphic) modular cuspforms of weight 2 for $\Gamma$ with trivial character.

The cohomology with compact support $H^1_c(Y(\Gamma),\C)$ is canonically isomorphic to the relative cohomology 
$H^1(Y(\Gamma), \partial Y(\Gamma),\C)$.  Its image in $H^1(Y(\Gamma),\C)$
is denoted $H^1_!(Y(\Gamma),\C)$ and called ``interior cohomology''.  It is well-known and not hard to see that $H^1_!(Y(\Gamma),\C)$ is naturally isomorphic to what is called ``parabolic cohomology'' in \cite{S}.

We have a map
\[S_2(\Gamma)\oplus \overline{S_2(\Gamma)}\to H^1_c(Y(\Gamma),\C)= H^1(Y(\Gamma),\partial Y(\Gamma), \C),\]
given by sending $f\oplus \overline g$ to the class that pairs with a cycle $\xi$ in $H_1(Y(\Gamma),\partial Y(\Gamma), \C)$ to give the value
\[\int_\xi f(z)dz+\overline g(z)d\overline z.\]  Composing this map with the defining map from $H^1_c(Y(\Gamma),\C)$ to $H^1_!(Y(\Gamma),\C)$ gives a map
\[a\colon S_2(\Gamma)\oplus \overline{S_2(\Gamma)}\to
H^1_!(Y(\Gamma),\C).\]

\begin{lemma}\label{ES}
Let $\Gamma\subset \SL_2(\Z)$ be a subgroup of finite index.
Then the map $a$ is an isomorphism.
\end{lemma}
\begin{proof}
This is a special case of a theorem of Eichler and Shimura.  It follows, for example,  
immediately from Deligne's statement~\cite[Th\'eor\`eme~2.10]{deligne}.
\end{proof}

Let $\lambda \colon H_1(Y(\Gamma), \partial Y(\Gamma),\C)\to H^1(Y(\Gamma),\C)$ be the isomorphism given by Lefschetz duality.

\begin{lemma}
 Let $\Gamma\subset \GL_2(\Z)$  be a subgroup of finite index.
 \begin{enumerate}
 \item \label{it:cusp-a} If $\Gamma\subset \SL_2(\Z)$ then 
  $a\inv\circ\lambda\circ b$ induces an isomorphism 
 \[\phi \colon H_0^{\cusp}(\Gamma, \St(\Q^2;\C))\to 
  S_2(\Gamma)\oplus \overline{S_2(\Gamma)}.\]
     
\item \label{it:cusp-b} If $J=\diag(1,-1)\in\Gamma$, let $\widetilde\Gamma=\Gamma\cap\SL_2(\Z)$.
Then  
$H_0^{\cusp}(\Gamma, \St(\Q^2;\C))$ has the same dimension as $S_2(\widetilde\Gamma)$.
 \end{enumerate}
\end{lemma}

\begin{proof}
 First suppose that $\Gamma\subset \SL_2(\Z)$.  Then the rows of the following commutative diagram is exact:
\[
\xymatrix{
H^1_c(Y(\Gamma),\C)\ar[r]
&H^1(Y(\Gamma),\C)\ar[r]^\alpha&H^1(\partial Y(\Gamma),\C)
\ar[r]&H^2_c(Y(\Gamma),\C)
\\
&H_0(\Gamma, \St(\Q^2;\C)) \ar[u]\ar[r]^\beta
&H_0(\partial Y(\Gamma),\C) \ar[u]
\simeq\DD 
}
\]
Here, $\alpha$ is induced by restriction of a cocycle to the boundary, and $\beta$ is the map that sends 
$[v,w]_\Gamma$ to $(w)_\Gamma-(v)_\Gamma$.
The vertical maps are isomorphisms: the one on the left is $\lambda\circ b$ and the one on the right is Poincar\'e duality, 
where we identify $\DD$ with 
$H_0(\partial Y(\Gamma),\C)$ in the obvious way.  
The kernel of $\alpha$ is $H^1_!(Y(\Gamma),\C)$
and the kernel of $\beta$ is $H_0^{\cusp}(\Gamma, \St(\Q^2;\C))$.
Together with Lemma~\ref{ES}, this proves part~\ref{it:cusp-a}.

To see part~\ref{it:cusp-b}, we first show that $\phi$  is $J$-equivariant up to sign. 
We begin by making explicit the isomorphism
 $a\inv\circ \lambda\circ b|H_0^{\cusp}(\Gamma, \St(\Q^2;\C))$.  The resulting integral formula is useful in its own right.  
 
 Without loss of generality, by Lemma~\ref{hcusp}, consider a cuspidal modular symbol 
 $[v,\gamma v]_{\widetilde\Gamma}$, where
 $\gamma\in\widetilde\Gamma$ and let  $w=\gamma v$.
Choose $z_v$ and then choose $z_w=\gamma z_v$ and then
choose a path in the upper half plane from $z_v$ to $z_w$ and project to the closed path in $Y(\widetilde\Gamma)$ whose fundamental class is $\xi$ .  Then $[v,w]_{\widetilde\Gamma}$ maps to the de Rham class of the closed differential $\omega_\xi$, where  $\omega_\xi$ satisfies 
\[
\int_\pi \omega_\xi =  \pair{\xi, \pi}
\]
for any closed path $\pi$ on $Y(\widetilde\Gamma)$.  Here $\pair{\cdot, \cdot}$ denotes the intersection pairing.  Then $\omega_\xi$ is cohomologous to 
$f_\xi(z)dz+\bar g_\xi(z)d\bar z$ for some $(f_\xi,\bar g_\xi)\in  S_2(\Gamma)\oplus \overline{S_2(\Gamma)}$.  Then
\[
(a\inv\circ \lambda\circ b)([v,\gamma v]_{\widetilde\Gamma})=(f_\xi,\bar g_\xi).
\]

Doing the same process to $[Jv,Jw]_{\widetilde\Gamma}$, we take $J\xi$, so the image of $[Jv,Jw]_{\widetilde\Gamma}$ is the 
de Rham class of the closed differential $\omega_{J\xi}$, where
\[
\int_\pi \omega_{J\xi} =  \pair{J\xi, \pi}
\]
for any $\pi$.
Since $J$ reverses orientation, $\pair{J\xi, \pi}=-\pair{\xi, J\pi}$.  Because
\[
\int_\pi J\omega = \int_{J\pi} \omega
\]
it follows that 
$\omega_{J\xi}=-J\omega_{\xi}$.  We see that the isomorphism $\phi$ in 
part~\ref{it:cusp-a} is $J$-equivariant up to sign, i.e. $\phi(Jx)=-J\phi(x)$.

If $f\in S_2(\widetilde\Gamma)$,
 $J(f(z)dz)=-\bar f(z) d\bar z$ (as may be seen for example from the $q$-expansion of $f$ and the fact that $J(q)=\bar q$ when $q=e^{2\pi iz}$.)
 If $\omega_\xi$ is cohomologous to 
$f_\xi(z)dz+\bar g_\xi(z)d\bar z$, then $-J\omega_{\xi}$
is cohomologous to 
$\bar f_\xi(z)d\bar z+ g_\xi(z)dz$.  So $-J$ acts on 
$S_2(\Gamma)\oplus \overline{S_2(\Gamma)}$ as complex conjugation, flipping the two factors.
 
By the homology version of \cite[Proposition III.10.4]{B}, 
$H_0^{\cusp}(\Gamma, \St(\Q^2;\C))$ is isomorphic to the $J$-coinvariants of $H_0^{\cusp}(\widetilde\Gamma, \St(\Q^2;\C))$.  The latter vector space has the same dimension as the $J$-invariants of $H_0^{\cusp}(\widetilde\Gamma, \St(\Q^2;R))$.  (For any finite dimensional $\C[\gen{J}]$-module $V$, write $V$ as a direct sum of irreducible components.  The trivial module has both $J$-invariants and $J$-coinvariants of dimension 1, and 
the nontrivial module has both $J$-invariants and $J$-coinvariants of dimension 0.)

  Under the isomorphism $\phi$, 
  the  $J$-invariants of $H_0^{\cusp}(\widetilde\Gamma, \St(\Q^2;R))$
map to the $J$-anti-invariants of 
$S_2(\widetilde\Gamma)\oplus \overline{S_2(\widetilde\Gamma)}$, i.e. 
$\{y\in S_2(\widetilde\Gamma)\oplus \overline{S_2(\widetilde\Gamma)} \ | \ J(y)=-y\}$.  
Since $J$ interchanges the two summands  $S_2(\widetilde\Gamma)$ and 
$\overline{S_2(\widetilde\Gamma)}$, we conclude that the dimension of 
$H_0^{\cusp}(\Gamma, \St(\Q^2;\C))$ equals the dimension of 
$S_2(\widetilde\Gamma)$.
\end{proof}

\noindent Let $g_0(N)$ be the genus of the modular curve $X_0(N)$. Thus
$\dim_\C S_2(\Gamma_0(N))=g_0(N)$.

\begin{corollary}\label{dims}
  If $N\ge1$, then
  \begin{align*}
\dim_\C  H_0^{\cusp}(\Gamma_0(N), \St(\Q^2;\C)) &=2g_0(N) \quad \text{and}\\ 
\dim_\C H_0^{\cusp}(\Gamma_0(N)^\pm, \St(\Q^2;\C)) & =g_0(N).
  \end{align*}
\end{corollary}
\noindent These are the dimensions we observe in our computations.

Now fix a quadratic field $E$ and  $\Gamma$ a congruence subgroup of $\SL_2(\Z)$.  
 B.~Gross and A.~Venkatesh have independently suggested to us that it might be possible to prove the
surjectivity of $\psi_{\Gamma,E}$ over $\C$ (at least in some cases) unconditionally as follows.  

\begin{definition}  \label{toral1}
Let $\gamma$ be a $\beta$-unital element in 
$\Gamma\subset\GL_2(\Z)$ and $E=\Q(\beta)$.  Then we will say that $\gamma$ is an $E$-unital element.

The image
of $\psi_{\Gamma,E}$
 is spanned by the set of modular symbols $[v,\gamma v]_\Gamma$, where 
 $\gamma$ runs through all the $E$-unital elements in $\Gamma$.

Let $f\in S_2(\Gamma)$ be an eigenform for the Hecke algebra.  Let $\tau$ be a point in the upper halfplane.
We call the number
\[
\int_\tau^{\gamma \tau} f(z) \,dz
\]
an $E$-toral period (or just a toral period if $E$ is understood.)

We call the fundamental class in $H_1(X(\Gamma), \C)$ of the closed curve which is the projection to $X(\Gamma)$ of the geodesic from 
$\tau$ to $\gamma \tau$
an $E$-toral cycle (or just a toral cycle if $E$ is understood.)
\end{definition}

For simplicity, consider the case of toral periods and cycles 
for $\Gamma=\Gamma_0(N)$.
If $[v,\gamma v]_\Gamma\in 
H_0^{\cusp}(\Gamma_0(N), \St(\Q^2;\C))$ corresponds as above to the fundamental class $\xi$, then
$\psi_\Gamma([v,\gamma v]_\Gamma)$ is determined by the values of 
\[ 
\pair{\xi,\pi}=\int_\pi \omega_{\xi} =
\int_\xi \omega_{\pi},
\]
as $\pi$ varies among closed paths on $X_0(N)$.  

Suppose 
$\psi_{\Gamma,E}$ is not surjective.  It is easy to see that the image of $\psi_{\Gamma,E}$ is Hecke-equivariant.  It follows that there is a cuspidal eigenform $f$ of level $N$ and weight 2 for the Hecke algebra such that 
\[
\int_{\xi} f(z) \, dz = 0
\]
for any 
 $\xi$ which is the fundamental class of the projection to $X(\Gamma)$ of the geodesic $\ell$ from $z_v$ to $z_w$, where $\gamma z_v = z_w$, and $\gamma$ is an $E$-unital element in 
 $\Gamma$.  
 Consider the geodesic $\mu$ between the two fixed points of 
$\gamma$ in $\R$ and fix an interior point $\tau$ in $\mu$.  
We may choose a $\Gamma$-equivariant homotopy in $\bar Y$ from $\ell$  to the arc from $\tau$ to $\gamma \tau$ in $\mu$.  Let $\alpha$ be the closed curve on $Y(\Gamma)$ which is the image of this arc.  Then the class of $\alpha$ is an $E$-toral cycle.
Because $f(z)\, dz$ is a closed form, 
\[
\int_{\xi} f(z) \, dz = \int_{\alpha} f(z) \, dz,
\]
which is an $E$- toral period.  These periods were studied by Waldspurger in \cite{W} and by various other authors.

The formulas of Waldspurger \cite{W} in principle may 
be made explicit so that, given certain characters $\chi$, 
we can write $L(f,\chi,1/2)$ as a linear combination of some of these toral periods.  If one has a theorem that shows that the $L$-value is nonvanishing for some relevant $\chi$, it follows that at least one of the periods is nonzero.  This contradiction would establish the surjectivity of
 $\psi_{\Gamma,E}$ without having to assume GRH.

Successfully implementing this idea does not look easy.  The only case we know where Waldspurger's formula has been made sufficiently explicit so that the preceding program might be carried out, is Theorem 6.3.1 in \cite{P}.  But even in this case, (where we must assume that $N$ is squarefree and coprime to the discriminant of $E$,  all the primes dividing $N$ split in $E$ and  $\chi$ is a character of the narrow class group of $E$)
 we do not know of any sufficiently strong non-vanishing theorem for the central value of the $L$-function.  
  
However, we can use the preceding ideas to prove the following theorem, which we have not seen in the literature.  
  \begin{theorem}\label{toralgen}
  Assume GRH.  Let $\Gamma$ be a congruence subgroup of $\SL_2(\Z)$.
  Then,  for any real quadratic field $E$, the $E$-toral cycles generate $H_1(X(\Gamma),\C)$.
  \end{theorem}
  
   \begin{proof}
Let $\gamma$ be an $E$-unital element of $\Gamma$.
The isomorphism 
\[b \colon H_0(\Gamma,\St(\Q^2;\C)) \to
H_1(Y(\Gamma), \partial Y(\Gamma),\C)\] from Definition~\ref{toral0}
sends $[e,\gamma e]_\Gamma$ to the fundamental class of the image modulo 
$\Gamma$ of the geodesic from $\hat z_e$ to $\hat z_{\gamma e}$, which is homologous to an $E$-toral cycle corresponding to $\gamma$.  By Theorem~\ref{maingeneral}, under GRH, we know that $\psi_{\Gamma,E}$ maps surjectively onto the cuspidal Steinberg homology, which is mapped by $b$ onto $H_1(X(\Gamma),\C)\subset H_1(Y(\Gamma), \partial Y(\Gamma),\C)$.  
\end{proof}

 For more information about toral periods, see \cite{shintani, osullivan-taylor} and the articles referenced there.

\section{Method of computation} \label{what}

We made the computations using already-existing programs that find the Voronoi homology of arithmetic subgroups of $\GL_n(F)$ for arbitrary number fields $F$.  For this reason, we have to show carefully that when $n=2$ and $F=\Q$ that this Voronoi homology is exactly isomorphic to $H_0^{\cusp}(\Gamma, \St(\Q^2;R))$.
We also need to make this isomorphism explicit, so that we know how to express the image
of $\psi_\Gamma$ in terms of the Voronoi homology.

We do these things in this section and also provide a numerical example of the computation of  the cuspidal Steinberg homology $H_0^{\cusp}(\Gamma, \St(\Q^2;R))$.
For general background on the Voronoi decomposition and Voronoi homology see \cite{voronoi1, PerfFormModGrp}.

Let $\Gammabar = \GL_2(\Z)$.  We will describe the computational method for the groups 
$\Gamma_0(N)^\pm$.  The method for $\Gamma_0(N)$ is similar.

Let $V$ be the $3$-dimensional vector space of real symmetric $2
\times 2$ matrices.  Let $C \subset V$ denote the open cone of
positive definite matrices, $\overline C$ its closure in $V$.  Let $q \colon \Z^2 \to V$ be the map
$q(v) = v v^t$.  (All vectors in this section are column vectors.)  Let $X =
C/\R_+$ and $\overline X=\overline C/\R_+$,
where $\R_+$ acts on $C$ by scaling.  The group
$\Gammabar$ acts on $V$: for $\gamma \in \Gammabar$ and $A \in V$,
$\gamma \cdot A = \gamma A \gamma^t$.  This action restricts
to an action on $\overline C$, which descends to an action on $\overline X$.

There is one $\Gammabar$-orbit of perfect forms.  As a representative,
take the one with minimal vectors given by $\pm$ the columns of 
\[
\begin{bmatrix}
1 & 0 & 1\\
0 & 1 & 1
\end{bmatrix}.
\]  
The image under $q$ of these minimal vectors
are the vertices of an ``ideal'' triangle in $X$.  
This is the ideal $\{\infty, 0,
1\}$-triangle in the complex upper half-plane $\hp$ under the usual identification of $X$ with $\hp$.  It is ideal in
the sense that 
the vertices are not in $X$, but in $\overline{X}$.    The
$\Gammabar$-orbit of this triangle covers $X$, yielding an
tessellation of $X$ by ideal triangles minus their vertices. Let $X^*$
denote $X$ union the vertices, so that we have an honest tessellation
of $X^*$.  The triangles are $2$-cells.  $X^*$ corresponds to $\hp \cup
\PP^1(\Q)$.  Each triangle has finite 
stabilizer in $\Gammabar$.

There is one equivalence class of edge in this tessellation.  As
representative, we can take the edge between $q(e)$ and $q(f)$.
 In $\hp$, this is 
the geodesic from $\infty$ to $0$.
Each edge has finite stabilizer in $\Gammabar$.  The edges are
$1$-cells. 

There is one equivalence class of vertex in this tessellation.  As
representative, we can take $q(e)$.  
This is the point $\infty$ in 
$\PP^1(\Q)$. The vertices are $0$-cells.
This is the point $\infty$ in 
$\PP^1(\Q)$.  The stabilizer of each vertex in $\Gammabar$ 
is infinite.  

Fix vectors $v_1 = e_1$, $v_2 = e_2$, $v_3 = e_1 + e_2$.  Denote the
representative cell by the vectors defining the cell, using the
subscript.  For example, $\{1,2,3\}$ is the representative triangle, $\{1,2\}$ is the
representative edge, and $\{1\}$ is the representative vertex.

The stabilizer of $\{1,2,3\}$ is 
\begin{multline*}
\Gammabar_2:= \Gammabar_{\{1,2,3\}} = \left\{
\mat{ -1& 1\\ -1& 0 },
\mat{ 1& 0\\ 0& 1 },
\mat{ 0& 1\\ 1& 0 },
\mat{ 1& -1\\ 0& -1 },
\mat{ -1& 1\\ 0& 1 },
\mat{ 0& 1\\ -1& 1 },
\mat{ 1& -1\\ 1& 0 },\right.\\
\left.\mat{ -1& 0\\ 0& -1 },
\mat{ 0& -1\\ -1& 0 },
\mat{ 0& -1\\ 1& -1 },
\mat{ -1& 0\\ -1& 1 },
\mat{ 1& 0\\ 1& -1 }
\right\}.
\end{multline*}
The orientation preserving stabilizer of $\{1,2,3\}$ is 
\[
\Gammabar_2^+:=
\Gammabar^+_{\{1,2,3\}} = \set*{
\mat{ -1& 1\\ -1& 0 },
\mat{ 1& 0\\ 0& 1 },
\mat{ 0& 1\\ -1& 1 },
\mat{ 1& -1\\ 1& 0 },
\mat{ -1& 0\\ 0& -1 },
\mat{ 0& -1\\ 1& -1 }
}.
\]

The stabilizer of $\{1,2\}$ is 
\begin{multline*}
\Gammabar_1:=
\Gammabar_{\{1,2\}} = \left\{
\mat{ -1& 0\\ 0& -1 },
\mat{ 0& -1\\ 1& 0 },
\mat{ 1& 0\\ 0& -1 },
\mat{ 0& -1\\ -1& 0 },
\mat{ 0& 1\\ 1& 0 },\right.\\
\mat{ 1& 0\\ 0& 1 },
\left. \mat{ -1& 0\\ 0& 1 },
\mat{ 0& 1\\ -1& 0 }
\right\}.\end{multline*}
The orientation preserving stabilizer of $\{1,2\}$ is 
\[
\Gammabar_1^+:=
\Gammabar^+_{\{1,2\}} = \set*{
\mat{ -1& 0\\ 0& -1 },
\mat{ 1& 0\\ 0& -1 },
\mat{ 1& 0\\ 0& 1 },
\mat{ -1& 0\\ 0& 1 }
}.\]

The stabilizer of $\{1\}$ is
\[
\Gammabar_0:=
\Gammabar_{\{1\}} = \set*{
\mat{ \pm1& *\\ 0& \pm1 }
}.\]
Every element of $\Gammabar_0$ is deemed to preserve orientation.

Let $\Gamma$ be a subgroup of
finite index in $\Gammabar$.  
For $i = 0, 1, 2$, let $\Sigma_i$ denote the $i$-cells in the
tessellation of $X^*$.  Then $\Gamma$ acts on $\Sigma_i$, and we
compute representatives of the $\Gamma$-orbits as follows.

For example, $\Sigma_2$ is the set of $\Gammabar$ translates of the
triangle $\{1,2,3\}$, and by definition
$\Gammabar_2$ stabilizes the cell.  Thus the
$\Gamma$-orbits of $\Sigma_2$ are parametrized by $\Gamma \backslash
\Gammabar / \Gammabar_2$.
Similarly, the
$\Gamma$-orbits of $\Sigma_1$ are parametrized by $\Gamma \backslash
\Gammabar / \Gammabar_1$
and the
$\Gamma$-orbits of $\Sigma_0$ are parametrized by $\Gamma \backslash
\Gammabar / \Gammabar_0$.

Now let $\Gamma=\Gamma_0(N)^\pm$ and $i=0,1,2$.
Since $\Gammabar$ acts transitively on $\PP^1(\Z/N\Z)$ from the
right, and the stabilizer of $[0:1]$ is $\Gamma$,  it follows that 
$\Gamma \backslash \Gammabar$ can be identified with
$\PP^1(\Z/N\Z)$, so $\Gamma \backslash
\Gammabar / \Gammabar_i$ is the set of
$\Gammabar_i$-orbits on $\PP^1(\Z/N\Z)$.
In other words, each $\Gammabar_i$-orbit in
$\PP^1(\Z/N\Z)$ can be identified
with a $\Gamma$-orbit in $\Sigma_i$.  There are finitely many such
orbits.  For each one, we fix an element $a$ in $\PP^1(\Z/N\Z)$ to
represent that orbit.   

If  one assumes that multiplication by 2 is an injective map $R\to R$ and
if for $i=0,1$ and $a$ in $\PP^1(\Z/N\Z)$
there exists an orientation reversing
element $\gamma \in \Gammabar_i - \Gammabar^+_i$ such that $a \cdot \gamma = a$, then the corresponding cell is
non-orientable in $X^*/\Gamma$, and we will see from the proof of 
Theorem~\ref{whatwhat} that we can (and must) remove it from the chain complex in order to compute the cuspidal Steinberg homology.  

Again let $\Gamma$ be an arbitrary subgroup of finite index in $\Gammabar$.  For $i=0,1,2$, 
let $\Sigma_i(\Gamma)$ denote the set
of $\Gamma$-orbits of vertices, edges and triangles respectively.  

For any subgroup $\Delta$ of $\Gammabar$,  and a homomorphism $\chi:\Delta\to R^\times$, 
Let $I(\Gammabar,\Delta,\chi)$ denote the induced left $\Gammabar$-module consisting of functions $f:\Gammabar\to R$ such that $f(gd)=\chi(d)f(g)$ for every $g\in\Gammabar$ and $d\in\Delta$. The action of $\Gammabar$ 
is given by $(xf)(g)=f(x\inv g)$.

For $i=0,1,2$ let $\chi_i:\Gammabar_i\to R^\times$ be the orientation character:  
$\chi_i(x)=1$ if $x$ preserves the orientation of the basic cell 
${\{1\}},{\{1,2\}},{\{1,2,3\}}$ respectively, and $\chi_i(x)=-1$ if $x$ reverses the orientation. In particular, $\chi_0$ is the identity character.

For $i=0,1,2$, set $\I_i(\Gamma)=I(\Gammabar, \Gammabar_i,\chi_i)$.  We have distinguished elements of these induced modules, namely $F_i$, which is defined to equal $\chi_i$ on $\Gammabar_i$ and 0 otherwise.

The boundary maps that send an oriented triangle to the sum of its oriented edges and an edge to the difference of its vertices induce $R$-module maps
$\partial_2:\I_2\to \I_1$ and $\partial_1:\I_1\to  \I_0$.  We define these maps carefully as follows.

First, the boundary of the triangle $(123)$ is $(12)+(23)+(31)$.  This triangle corresponds in the induced representation to the function $F_2$.
Define the matrices
\[
U=\begin{bmatrix}
0&1\\
-1&1
\end{bmatrix},\ \ 
V=\begin{bmatrix}
1&-1\\
1&0
\end{bmatrix}.
\]
Then $\partial_2(F_2)=(1+U+V)F_1$.  Since the boundary map is $\Gammabar$-equivariant, for any $g\in\Gammabar$, $\partial_2(gF_2)=g(1+U+V)F_1$.

Second, the boundary of the  edge $(12)$ is $(2)-(1)$.  This edge corresponds in the induced representation to the function $F_1$.
Define the matrix
\[
S=\begin{bmatrix}
0&1\\
-1&0
\end{bmatrix}.
\]
Then $\partial_1(F_1)=(S-1)F_0$.  Since the boundary map is $\Gammabar$-equivariant, for any $g\in\Gammabar$, $\partial_1(gF_1)=g(S-1)F_0$.

Let $\Gamma$ be a subgroup of finite index in $\Gammabar$ and let $A^\Gamma$ denote the submodule of $\Gamma$-invariants in any $\Gammabar$-module $A$.
For $i=0,1,2$ set  $I_i(\Gamma)=\I_i^\Gamma$.

The boundary maps defined above are $\Gammabar$-equivariant, and therefore we get a complex:
\[
\xymatrix{
I_2(\Gamma)\ar[r]^{\bar\partial_2}&I_1(\Gamma)\ar[r]^{\bar\partial_1}&I_0(\Gamma).
}
\]

What we compute by computer is the homology of this complex, also called ``the cuspidal Voronoi homology.''  We will show that it is naturally isomorphic to the cuspidal Steinberg homology of $\Gamma$.  In the statement and proof of this theorem we will use the notation $[g]$ for a modular symbol, where $g\in M_2(\Q)$ with neither column the 0-vector, and where $[g]$ stands for the modular symbol $[ge,gf]$ (where as usual $e=(1:0), f=(0:1)\in \PP^2(\Q)$).
 
\begin{theorem}\label{whatwhat}
Let $R$ be an integral domain.
Let $\Gamma$ be a subgroup of finite index in $\Gammabar$ and define the modules and maps $\bar\partial_2 \colon I_2(\Gamma)\to I_1(\Gamma)$ 
and $\bar\partial_1 \colon I_1(\Gamma)\to I_0(\Gamma)$ as in the preceding discussion.
Then there is a natural isomorphism 
\[\Ker(\bar\partial_1)/\image \bar\partial_2 \to H_0^{\cusp}(\Gamma,\St(\Q^2;R)).\]
This isomorphism takes the class of $F\in \Ker(\bar\partial_1)$
 to $\sum_{g\in\cRR_1}F(g)[g]_\Gamma$, where $\cRR_1$ denotes a set of representatives of the double cosets 
 $\Gamma\backslash \Gammabar/ \Gammabar_1$.
\end{theorem}
 
 \begin{proof}  Many of the details of this proof consist of tedious checking and will be omitted.
 For $i=0,1,2$, write $I_i$ instead of $I_i(\Gamma)$, for brevity.  Let $\cRR_i$ denote a set of representatives of the double cosets 
 $\Gamma\backslash \Gammabar/ \Gammabar_i$.
 Let $J$ denote the free $R$-module on $\Gammabar$.  
 If $g\in\Gammabar$ we write $g=(u,v)$ where $u$ and $v$ are the columns of $g$.  
 
 Let $A$ be the submodule of $J$ generated by the relations $(u,v)-(\pm u,\pm v)$ and $(u,v)+(v,u)$, where all four choices of $\pm$ are used.  Note that these relations are those generated by $\Gammabar_1$ acting on $J$ on the right,
 taking the character $\chi_1$ into account.
 Let $B$ be the submodule of $J$ generated by the relations 
 $(u,v)-(\gamma u, \gamma v)$, where $\gamma$ runs over $\Gamma$.
 
 Define an $R$-module map 
\[
 \theta:I_1\to J/(A+B)
\]
 by $\theta(F)=\sum_{g\in\cRR_1} F(g)g$ mod $(A+B)$.   The first task is to prove that 
 $\theta$ is an isomorphism of $R$-modules and does not depend on the choice of double coset representatives $\cRR_1$.   A key point here is to distinguish between oriented and unoriented objects:
 
 For $i=0,1,2$ let us say that $g\in\Gammabar$ is ``$i$-unoriented'' if and only if there exist 
$\gamma\in\Gamma$ and $\delta\in \Gammabar_1$ such that $\gamma g\delta = g$ and $\chi_i(\delta)\ne 1$.  Otherwise, say that $g$ is ``$i$-oriented.''
The property of being $i$-oriented or $i$-unoriented depends only on the double coset
$\Gamma g \Gammabar_i$, and we call the double coset $i$-oriented or $i$-unoriented accordingly.  Let $\cRR_i^*$ be a set of representatives for the $i$-oriented double cosets.

Given an $i$-orientable $g$, define $F_g^i\in I_i$  to be the function 
\[
F_g^i(\gamma g \delta) = \begin{cases}
  \chi_i(\delta) & \text{for $\gamma\in\Gamma$ and $\delta\in \Gammabar_i$,}\\   0 & \text{otherwise.}
 \end{cases}\] 
Then we check that any $F\in I_i$ is supported on the union of the orientable double cosets,
that the set of functions $\{F_g^i \ | \ g\in\cRR_i^*\}$ is a free $R$-basis of $I_i$, 
and that $\theta$ is an isomorphism.
We omit the details.
 
We now have to determine 
$\theta(\image(\bar\partial_2))$.  Let $C$ be the submodule of $J$ generated by $g(1+U+V)$ as $g$ runs over $\Gammabar$.  Our claim is that 
$\theta(\image(\bar\partial_2))$ is congruent to $C$ modulo $A+B$.  
The key point here is what we observed earlier, that $\partial_2(F_2)=(1+U+V)F_1$.
We omit the details.

By Corollary~\ref{triv},
 $J/(A+B+C)$ is naturally isomorphic to $H_0(\Gamma,\St(\Q^2;R))$ via the map $x\mapsto [x]_\Gamma$.  So we see that $\theta$ induces an isomorphism 
$\bar\theta:I_1/\image\bar\partial_2\to H_0(\Gamma,\St(\Q^2;R))$, and this verifies the explicit form of the isomorphism asserted in the theorem.

It remains to show that $\bar\theta$ takes the kernel of $\bar\partial_1$ to
$H_0^{\cusp}(\Gamma,\St(\Q^2;R))$.  
This can be done using Definition~\ref{defcusp}
and the fact (observed earlier) that $\partial_1(F_1)=(S-1)F_0$.  

\end{proof}

\begin{example}
In this example, we assume that multiplication by 2 is injective on the coefficient ring $R$.  This allows us to use a much smaller chain complex, because there are many unorientable cells.
If there exists an orientation reversing element 
$\gamma \in \Gammabar_{\{1,2,3\}} \setminus \Gammabar^+_{\{1,2,3\}}$ 
such that $a \cdot \gamma = a$, then that triangle is
non-orientable in $X^*/\Gamma$, and we remove it from the cell complex.  Similar we remove non-orientable edges from the cell complex.

Let $N = 11$ and $\Gamma=\Gamma_0(11)^\pm$.  
The finite projective space $\PP^1(\F_{11})$ has $12$ points:
$[0:1], [1:0], [1:1], \dots, [1:10]$.
We represent any given cell by a matrix, whose columns determine the vertices
of the cell.

Using the stabilizer $\Gammabar_{\{1,2,3\}}$, the points of $\PP^1(\F_{11})$ get grouped into 3 orbits:
\begin{align*}
t_1 &=\{[0:1],[1:0], [1:10]\}\\
t_2 &=\{[1:1],[1:5],[1:9]\}\\
t_3 &=\{[1:2],[1:3],[1:4],[1:6],[1:7],[1:8]\}.
\end{align*}
Only $t_3$ is orientable, so $C_2$ is 1-dimensional.
We choose as a representative $2$-cell $\sigma = \mat{0 & -1 & -1\\1 & 2 & 3}$.

Using the stabilizer $\Gammabar_{\{1,2\}}$, the points of $\PP^1(\F_{11})$ get grouped into 4
orbits:
\begin{align*}
e_1 &= \{[0:1],[1:0]\}\\
e_2 &= \{[1:1],[1:10]\}\\
e_3 &= \{[1:2],[1:5],[1:6],[1:9]\}\\
e_4 &= \{[1:3],[1:4],[1:7],[1:8]\}.
\end{align*}
We have $e_1$, $e_3$, and $e_4$ are orientable, so $C_1$ is $3$-dimensional.
We choose as representatives
\[\tau_1 = \mat{1 & 0\\0 & 1},\quad  \tau_2 = \mat{0 & -1\\1 &
  2},\quad\text{and} \quad   \tau_3 = \mat{0 & -1 \\ 1 & 3}.\]

We now compute the boundary map $\partial_2 \colon C_2 \to C_1$, 
\[\partial_2(\sigma) = \mat{-1 & -1\\2 & 3} - \mat{0 & -1\\1 & 3} + \mat{0 &-1\\1
  & 2}= 
\mat{-1 & -1\\2 & 3} - \tau_3 + \tau_2.\]
Since 
\[\mat{-1 & -1\\2 & 3} = \mat{-4 & -1\\11 & 3} \mat{0 & -1\\1 & 3} \mat{0 &
  1\\-1 & 0},\]
it follows that the cell $\mat{-1 & -1\\2 & 3}$ is $-\tau_3$.  Thus $\partial_2
(\sigma) = \tau_2 - 2\tau_3$, and so the matrix representing this boundary
operator is $[\partial_2] = \mat{0\\1\\-2}$.

Using the stabilizer $\Gamma$, the points of $\PP^1(\F_{11})$ get grouped into 2
orbits:
\begin{align*}
v_1 &= \{[0:1],[1:1],[1:2], \dots, [1:10]\}\\
v_2 &= \{[1:0]\}.
\end{align*}
It follows that $C_0$ is $2$-dimensional.  
We choose as representatives $\rho_1 = \mat{1\\0}$ and $\rho_2 =
\mat{0\\1}$.  We now compute the boundary map $\partial_1 \colon C_1
\to C_0$.  
\begin{align*}
\partial_2(\tau_1) &= \mat{0\\1} - \mat{1\\0} = \rho_2 - \rho_1,\\
\partial_2(\tau_2) &= \mat{-1\\2} - \mat{0\\1} = \mat{-1\\2} - \rho_2,\\
\partial_2(\tau_3) &= \mat{-1\\3} - \mat{0\\1} = \mat{-1\\3}- \rho_2.
\end{align*}
Since 
\[
\mat{-1\\2} = \mat{6 & -1\\11 & 2} \mat{0\\1}\quad \text{and} \quad 
\mat{-1\\3} = \mat{4 & -1\\11 & 2} \mat{0\\1},
\]
we have that $\mat{-1\\2}$ and $\mat{-1\\3}$ are both $\rho_2$ so
\[\partial_1(\tau_2) =  \partial_1(\tau_3) = 0.\]
Thus the matrix representing this boundary operator is $[\partial_1] =
\mat{-1 & 0 & 0\\ 1 & 0 & 0}$.

For the homology computation, the kernel of $\partial_1$ is generated
by $\tau_2$ and $\tau_3$, and in the quotient by the image of
$\partial_2$, we have $[\tau_2] = 2 [\tau_3]$.  Here, we use square
brackets to signify a homology class.
\end{example}

\section{Remarks on the computations and on tables of results}\label{results}
We performed computations to find the image of $\psi_{\Gamma,E}$, when $R=\Z$, for 
$\Gamma$ equal to 
$\Gamma_0(N)^\pm$ and $\Gamma_0(N)^\pm$ for $N\le1000$ and $E=\Q[\sqrt\Delta]$ with 
$\Delta\le50$.  There are \num{29610} pairs $[N, \Delta]$ for which the Steinberg homology is nontrivial.  All of the computations are done in \cite{magma}, with some processing of the data done with SageMath \cite{sagemath}.  We first computed the cuspidal Steinberg homology groups over $\Z$ as described in Section~\ref{what}.  Then 
for each pair $[N,\Delta]$, we computed the image of $\psi_{\Gamma,E}$ as follows.  

From Corollary~\ref{cor:im_psi}, the image is generated by the symbols of the form $[e, \gamma_\beta e]_\Gamma$, for $(\beta:1) \in \PP^1(E)\setminus \PP^1(\Q)$.  We generate $\beta$-unital matrices by looking at higher and higher powers of the fundamental unit $\epsilon$ of $E$, and finding $\beta$ such that $\gamma_\beta$ is in $\Gamma$.  
For each power of $\epsilon$, we find several $\beta$ before moving on to the next power of $\epsilon$.  This turned out to be more efficient than first looping through $\beta$'s of growing height and for each $\beta$ finding the smallest $k$ such that 
 $\rho_\beta(\epsilon^k)$ is in $\Gamma$.
  For each $\gamma_\beta$, the symbol $[e, \gamma_\beta e]_\Gamma$ is computed using the usual continued fractions technique for modular symbols \cite[page~14]{cremona}.

  For each pair $[N,\Delta]$, the computation generates elements in the image of 
  $\psi_{\Gamma,E}$.  These elements generate a submodule of the image of 
  $\psi_{\Gamma,E}$.  There are quick exits if we find this submodule is equal to the whole cuspidal Steinberg homology, since in that case $\psi_{\Gamma,E}$ is surjective onto it.  Otherwise, the computation runs for 100 seconds.  We know this tabulated result is a subspace of the true image of $\psi_{\Gamma,E}$.  In practice, this subspace of the image of $\psi_{\Gamma,E}$ almost always stabilized quickly, in which case we record it as our output, and have confidence that it is the true image.  In the few cases where stabilization did not occur by 100 seconds, we computed further until we were satisfied that the result had stabilized.
  
Given $N$ and $E$, it is true that if we were to compute for more and more $\beta$'s, the reported image of $\psi_{\Gamma_0(N),E}$ could get bigger.  For instance, it is consistent with our calculations that that 
$\psi_{\Gamma_0(N),E}$ is always surjective.  We do not have an effective bound on the height of $\beta$ that would ensure the correctness of our reported image.  

The following facts lend additional credibility to our calculations.  Let $R(N,E)^\pm$ denote the reported image of $\psi_{\Gamma_0(N)^\pm,E}$, and
let $R(N,E)$ denote the reported image 
of $\psi_{\Gamma_0(N),E}$.
Then
\begin{enumerate}
\item The rank of $R(N,E)^\pm$ is always the genus $g(N)$ of the compact modular curve $X_0(N)$, and the rank of $R(N,E)$ is always $2g(N)$, consistent with the results of Section~\ref{complex}.

\item $\pi(R(N,E)^\pm)$ always contains $A_E(N)$ as a subgroup of index at most $4$, 
and $\pi^*(R(N,E))$ always contains $A_E^*(N)$ as a subgroup of index at most $4$, 
consistent with the results of Section~\ref{conj}.
\end{enumerate}

See Tables~\ref{tab:ane} and \ref{tab:SL-ane} in Section~\ref{tables} for a small sample of the computational results.  
The full set of results for level $N \leq 1000$ and real quadratic fields $\Q(\sqrt{\Delta})$ for $\Delta \leq 50$ are available online 
 (\url{https://mathstats.uncg.edu/yasaki/data/}).

\section{Tables of sample data}\label{tables}

In the following tables, the meaning of the column headings is as follows:

For $\Gamma=\Gamma_0(N)^\pm$:  $U^\pm = (\Z/N\Z)^\times/\set{\pm 1}$, 
$A^\pm = A_E(N)$, and  $Q^\pm = U^\pm/A^\pm$.   We also list the free rank $r^\pm$ and torsion subgroup $T^\pm$ of $H^\cusp_0(\Gamma_0(N)^\pm,\St(\Q^2))$, the cokernel $C^\pm$ of $\psi_{\Gamma,\Q(\sqrt\Delta)}$, and the shrinkage $s^\pm =\size{Q^\pm}/\size{C^\pm}$.  The
  \emph{$\Delta^\pm$-list} is a list of squarefree $\Delta \leq 50$, such
  that $\Q(\sqrt{\Delta})$ has the given information.

For $\Gamma=\Gamma_0(N)$: $U = (\Z/N\Z)^\times$, $A = A^*_E(N)$, and $Q = U/A$.   We also list the free rank $r$ and torsion subgroup $T$ of 
$H^\cusp_0(\Gamma_0(N),\St(\Q^2))$, the cokernel $C$ of $\psi_{\Gamma,\Q(\sqrt\Delta)}$,  and shrinkage $s =\size{Q}/\size{C}$.  

The
  \emph{$\Delta$-list} is a list of squarefree $\Delta \leq 50$, such
  that $\Q(\sqrt{\Delta})$ has the given information.
For $\Delta\in \{ 2, 5, 10, 13, 17, 26, 29, 37, 41\}$, Q($\sqrt \Delta$) has a unit of norm $-1$.
For $\Delta\in \{ 3, 6, 7, 11, 14, 15, 19, 21, 22, 23, 30, 31, 33, 34, 35, 38, 39, 42, 43, 46, 47 \}$, Q($\sqrt \Delta$) does not have a unit of norm $-1$.

Each table includes:  all the data for levels $N\le 20$, 
a few examples with shrinkage equal to $4$, and the cokernels found with largest cardinality.

Note: The cokernel is not always cyclic.  We have several examples of 
$(\Gamma_0(N)^\pm,E)$ where the cokernel has $2$ cyclic factors although none with $3$ or more cyclic factors.  For $(\Gamma_0(N),E)$, we have several examples with $3$ cyclic factors.  For instance, the cokernel for $\Gamma_0(840)$ is isomorphic to $C_2\times C_2\times C_6$ when 
$\Delta=37$.

\begin{table}
  \caption{Data for $\Gamma_0^\pm(N)$  \label{tab:ane}}
          {\small
            \begin{tabularx}{\linewidth}{r ccc ccccX}
\toprule
$N$ & $U^\pm$ & $A^\pm$ & $Q^\pm$ & $C^\pm$ & $r^\pm$ & $T^\pm$ & $s^\pm$ & $\Delta^\pm$-list  \\
\midrule
7 & $C_{3}$ & $C_{3}$ & $C_{1}$ & $C_{1}$ & $0$ & $C_{3}$ & $1$ & [2, 11, 15, 23, 29, 30, 37, 39, 43, 46]\\
\hline
7 & $C_{3}$ & $C_{1}$ & $C_{3}$ & $C_{3}$ & $0$ & $C_{3}$ & $1$ & [3, 5, 6, 7, 10, 13, 14, 17, 19, 21, 22, 26, 31, 33, 34, 35, 38, 41, 42, 47]\\
\hline
11 & $C_{5}$ & $C_{5}$ & $C_{1}$ & $C_{1}$ & $1$ & $C_{1}$ & $1$ & [3, 5, 14, 15, 23, 26, 31, 34, 37, 38, 42, 47]\\
\hline
11 & $C_{5}$ & $C_{1}$ & $C_{5}$ & $C_{5}$ & $1$ & $C_{1}$ & $1$ & [2, 6, 7, 10, 11, 13, 17, 19, 21, 22, 29, 30, 33, 35, 39, 41, 43, 46]\\
\hline
13 & $C_{6}$ & $C_{6}$ & $C_{1}$ & $C_{1}$ & $0$ & $C_{3}$ & $1$ & [3, 14, 17, 22, 23, 29, 30, 38, 43]\\
13 & $C_{6}$ & $C_{3}$ & $C_{2}$ & $C_{1}$ & $0$ & $C_{3}$ & $2$ & [10, 35]\\
\hline
13 & $C_{6}$ & $C_{2}$ & $C_{3}$ & $C_{3}$ & $0$ & $C_{3}$ & $1$ & [2, 5, 13, 26, 37, 41, 42]\\
13 & $C_{6}$ & $C_{1}$ & $C_{6}$ & $C_{3}$ & $0$ & $C_{3}$ & $2$ & [6, 7, 11, 15, 19, 21, 31, 33, 34, 39, 46, 47]\\
\hline
14 & $C_{3}$ & $C_{3}$ & $C_{1}$ & $C_{1}$ & $1$ & $C_{1}$ & $1$ & [2, 11, 15, 23, 30, 37, 39, 43, 46]\\
\hline
14 & $C_{3}$ & $C_{1}$ & $C_{3}$ & $C_{3}$ & $1$ & $C_{1}$ & $1$ & [3, 5, 6, 7, 10, 13, 14, 17, 19, 21, 22, 26, 29, 31, 33, 34, 35, 38, 41, 42, 47]\\
\hline
15 & $C_{4}$ & $C_{4}$ & $C_{1}$ & $C_{1}$ & $1$ & $C_{1}$ & $1$ & [6, 10, 11, 13, 19, 21, 31, 34, 37, 39, 46]\\
\hline
15 & $C_{4}$ & $C_{2}$ & $C_{2}$ & $C_{2}$ & $1$ & $C_{1}$ & $1$ & [2, 3, 5, 15, 17, 22, 23, 26, 29, 30, 35, 41, 42, 43, 47]\\
15 & $C_{4}$ & $C_{1}$ & $C_{4}$ & $C_{2}$ & $1$ & $C_{1}$ & $2$ & [7, 14, 33, 38]\\
\hline
17 & $C_{8}$ & $C_{8}$ & $C_{1}$ & $C_{1}$ & $1$ & $C_{1}$ & $1$ & [2, 13, 15, 21, 30, 33, 35, 42]\\
\hline
17 & $C_{8}$ & $C_{4}$ & $C_{2}$ & $C_{2}$ & $1$ & $C_{1}$ & $1$ & [26, 38, 43, 47]\\
\hline
17 & $C_{8}$ & $C_{2}$ & $C_{4}$ & $C_{4}$ & $1$ & $C_{1}$ & $1$ & [5, 10, 17, 19, 29, 37, 41]\\
17 & $C_{8}$ & $C_{1}$ & $C_{8}$ & $C_{4}$ & $1$ & $C_{1}$ & $2$ & [3, 6, 7, 11, 14, 22, 23, 31, 34, 39, 46]\\
\hline
19 & $C_{9}$ & $C_{9}$ & $C_{1}$ & $C_{1}$ & $1$ & $C_{3}$ & $1$ & [5, 6, 7, 17, 23, 26, 30, 35, 39, 42, 43]\\
\hline
19 & $C_{9}$ & $C_{3}$ & $C_{3}$ & $C_{3}$ & $1$ & $C_{3}$ & $1$ & [11, 47]\\
\hline
19 & $C_{9}$ & $C_{1}$ & $C_{9}$ & $C_{9}$ & $1$ & $C_{3}$ & $1$ & [2, 3, 10, 13, 14, 15, 19, 21, 22, 29, 31, 33, 34, 37, 38, 41, 46]\\
\hline
20 & $C_{4}$ & $C_{4}$ & $C_{1}$ & $C_{1}$ & $1$ & $C_{1}$ & $1$ & [5, 6, 13, 14, 17, 21, 34, 37, 39, 41, 46]\\
\hline
20 & $C_{4}$ & $C_{2}$ & $C_{2}$ & $C_{2}$ & $1$ & $C_{1}$ & $1$ & [2, 3, 7, 10, 11, 15, 19, 22, 23, 26, 29, 30, 31, 33, 35, 38, 42, 43, 47]\\
\midrule
\midrule
65 & $C_{2} \times C_{12}$ & $C_{3}$ & $C_{2} \times C_{4}$ & $C_{2}$ & $5$ & $C_{1}$ & $4$ & [35]\\
\hline
65 & $C_{2} \times C_{12}$ & $C_{1}$ & $C_{2} \times C_{12}$ & $C_{6}$ & $5$ & $C_{1}$ & $4$ & [7, 47]\\
\hline
285 & $C_{2} \times C_{36}$ & $C_{3}$ & $C_{2} \times C_{12}$ & $C_{6}$ & $37$ & $C_{1}$ & $4$ & [7]\\
\hline
285 & $C_{2} \times C_{36}$ & $C_{1}$ & $C_{2} \times C_{36}$ & $C_{18}$ & $37$ & $C_{1}$ & $4$ & [14]\\
\midrule
\midrule
983 & $C_{491}$ & $C_{1}$ & $C_{491}$ &   $C_{491}$& $82$ & $C_{1}$ & $1$ & [5, 10, 11, 13, 15, 17, 22, 26, 29, 30, 33, 34, 35, 39]\\
\hline
991 & $C_{495}$ & $C_{1}$ & $C_{495}$ &  $C_{495}$ & $82$ & $C_{3}$ & $1$ & [3, 6, 7, 11, 14, 15, 17, 22, 23, 30, 34, 35, 37, 39, 41, 46, 47]\\
\bottomrule
\end{tabularx}}
\end{table}

\begin{table}
  \caption{Data for $\Gamma_0(N)$ } \label{tab:SL-ane}
          {\small
            \begin{tabularx}{\linewidth}{r ccc ccccX}
\toprule
$N$ & $U$ & $A$ & $Q$ & $C$ & $r$ & $T$ & $s$ & $\Delta$-list  \\
\midrule
7 & $C_{6}$ & $C_{6}$ & $C_{1}$ & $C_{1}$ & $0$ & $C_{3}$ & $1$ & [2, 11, 15, 23, 29, 30, 37, 39, 43, 46]\\
\hline
7 & $C_{6}$ & $C_{2}$ & $C_{3}$ & $C_{3}$ & $0$ & $C_{3}$ & $1$ & [3, 5, 6, 7, 10, 13, 14, 17, 19, 21, 22, 26, 31, 33, 34, 35, 38, 41, 42, 47]\\
\hline
11 & $C_{10}$ & $C_{10}$ & $C_{1}$ & $C_{1}$ & $2$ & $C_{1}$ & $1$ & [3, 5, 14, 15, 23, 26, 31, 34, 37, 38, 42, 47]\\
\hline
11 & $C_{10}$ & $C_{2}$ & $C_{5}$ & $C_{5}$ & $2$ & $C_{1}$ & $1$ & [2, 6, 7, 10, 11, 13, 17, 19, 21, 22, 29, 30, 33, 35, 39, 41, 43, 46]\\
\hline
13 & $C_{12}$ & $C_{12}$ & $C_{1}$ & $C_{1}$ & $0$ & $C_{3}$ & $1$ & [3, 14, 22, 23, 30, 38, 43]\\
13 & $C_{12}$ & $C_{6}$ & $C_{2}$ & $C_{1}$ & $0$ & $C_{3}$ & $2$ & [10, 17, 29, 35]\\
\hline
13 & $C_{12}$ & $C_{4}$ & $C_{3}$ & $C_{3}$ & $0$ & $C_{3}$ & $1$ & [42]\\
13 & $C_{12}$ & $C_{2}$ & $C_{6}$ & $C_{3}$ & $0$ & $C_{3}$ & $2$ & [2, 5, 6, 7, 11, 13, 15, 19, 21, 26, 31, 33, 34, 37, 39, 41, 46, 47]\\
\hline
14 & $C_{6}$ & $C_{6}$ & $C_{1}$ & $C_{1}$ & $2$ & $C_{1}$ & $1$ & [2, 11, 15, 23, 30, 37, 39, 43, 46]\\
\hline
14 & $C_{6}$ & $C_{2}$ & $C_{3}$ & $C_{3}$ & $2$ & $C_{1}$ & $1$ & [3, 5, 6, 7, 10, 13, 14, 17, 19, 21, 22, 26, 29, 31, 33, 34, 35, 38, 41, 42, 47]\\
\hline
15 & $C_{2} \times C_{4}$ & $C_{2} \times C_{4}$ & $C_{1}$ & $C_{1}$ & $2$ & $C_{1}$ & $1$ & [6, 11, 19, 21, 31, 34, 39, 46]\\
\hline
15 & $C_{2} \times C_{4}$ & $C_{2} \times C_{2}$ & $C_{2}$ & $C_{2}$ & $2$ & $C_{1}$ & $1$ & [2, 3, 5, 10, 13, 15, 17, 22, 23, 26, 29, 30, 35, 37, 41, 42, 43, 47]\\
\hline
15 & $C_{2} \times C_{4}$ & $C_{2}$ & $C_{4}$ & $C_{4}$ & $2$ & $C_{1}$ & $1$ & [7, 14, 33, 38]\\
\hline
17 & $C_{16}$ & $C_{16}$ & $C_{1}$ & $C_{1}$ & $2$ & $C_{1}$ & $1$ & [15, 21, 30, 33, 35, 42]\\
\hline
17 & $C_{16}$ & $C_{8}$ & $C_{2}$ & $C_{2}$ & $2$ & $C_{1}$ & $1$ & [2, 13, 38, 43, 47]\\
\hline
17 & $C_{16}$ & $C_{4}$ & $C_{4}$ & $C_{4}$ & $2$ & $C_{1}$ & $1$ & [19, 26]\\
17 & $C_{16}$ & $C_{2}$ & $C_{8}$ & $C_{4}$ & $2$ & $C_{1}$ & $2$ & [3, 5, 6, 7, 10, 11, 14, 17, 22, 23, 29, 31, 34, 37, 39, 41, 46]\\
\hline
19 & $C_{18}$ & $C_{18}$ & $C_{1}$ & $C_{1}$ & $2$ & $C_{3}$ & $1$ & [5, 6, 7, 17, 23, 26, 30, 35, 39, 42, 43]\\
\hline
19 & $C_{18}$ & $C_{6}$ & $C_{3}$ & $C_{3}$ & $2$ & $C_{3}$ & $1$ & [11, 47]\\
\hline
19 & $C_{18}$ & $C_{2}$ & $C_{9}$ & $C_{9}$ & $2$ & $C_{3}$ & $1$ & [2, 3, 10, 13, 14, 15, 19, 21, 22, 29, 31, 33, 34, 37, 38, 41, 46]\\
\hline
20 & $C_{2} \times C_{4}$ & $C_{2} \times C_{4}$ & $C_{1}$ & $C_{1}$ & $2$ & $C_{1}$ & $1$ & [6, 14, 21, 34, 39, 46]\\
\hline
20 & $C_{2} \times C_{4}$ & $C_{2} \times C_{2}$ & $C_{2}$ & $C_{2}$ & $2$ & $C_{1}$ & $1$ & [2, 3, 5, 7, 10, 11, 13, 15, 17, 19, 22, 23, 26, 29, 30, 31, 33, 35, 37, 38, 41, 42, 43, 47]\\
\midrule
\midrule
65 & $C_{4} \times C_{12}$ & $C_{6}$ & $C_{2} \times C_{4}$ & $C_{2}$ & $10$ & $C_{1}$ & $4$ & [35]\\
\hline
65 & $C_{4} \times C_{12}$ & $C_{2}$ & $C_{2} \times C_{12}$ & $C_{6}$ & $10$ & $C_{1}$ & $4$ & [2, 5, 7, 13, 17, 37, 41, 47]\\
\hline
85 & $C_{4} \times C_{16}$ & $C_{2}$ & $C_{2} \times C_{16}$ & $C_{8}$ & $14$ & $C_{1}$ & $4$ & [5, 10, 17, 37, 41]\\
\hline
130 & $C_{4} \times C_{12}$ & $C_{6}$ & $C_{2} \times C_{4}$ & $C_{2}$ & $34$ & $C_{1}$ & $4$ & [35]\\
\hline
130 & $C_{4} \times C_{12}$ & $C_{2}$ & $C_{2} \times C_{12}$ & $C_{6}$ & $34$ & $C_{1}$ & $4$ & [2, 5, 7, 13, 17, 37, 41, 47]\\
\midrule
\midrule
983 & $C_{982}$ & $C_{2}$ & $C_{491}$ & $C_{491}$ & $164$ & $C_{1}$ & $1$ & [5, 10, 11, 13, 15, 17, 22, 26, 29, 30, 33, 34, 35, 39]\\
\hline
991 & $C_{990}$ & $C_{2}$ & $C_{495}$ & $C_{495}$ & $164$ & $C_{3}$ & $1$ & [3, 6, 7, 11, 14, 15, 17, 22, 23, 30, 34, 35, 37, 39, 41, 46, 47]\\
\bottomrule
\end{tabularx}}
\end{table}

\clearpage
\bibliographystyle{amsalpha}
\bibliography{biblio}

\end{document}